\documentclass[a4paper,10pt]{article}

\usepackage{amssymb} 
\usepackage{amsfonts}
\usepackage{amsmath}
\usepackage{amsthm}
\usepackage[utf8]{inputenc}
\usepackage{epsfig}
\usepackage{epstopdf}
\usepackage{pgf}%,pgfarrows}
\usepackage{graphicx}
\usepackage{tikz}
\usepackage{pdfpages}
\usepackage[all]{xy}
\usetikzlibrary{patterns}
\title{{\bfseries A tropical construction of a family of real reducible curves.}}
\author{\textbf{Arthur Renaudineau}}
\date{}

\newtheorem{theorem}{Theorem}
\newtheorem{lemma}{Lemma}
\newtheorem{remark}{Remark}

\newtheorem{example}{Example}
\newtheorem{definition}{Definition}
\newtheorem*{conjecture}{Conjecture}

\newcommand{\R}{\mathbb{R}}
\newcommand{\RP}{\mathbb{RP}}
\newcommand{\Z}{\mathbb{Z}}
\newcommand{\CP}{\mathbb{CP}}

\newcommand{\C}{\mathbb{C}}
\newcommand{\CC}{\mathbb{C}}
\newcommand{\T}{\mathbb{T}}

\begin{document}
\maketitle

% Commentaire : la commande \texorpdfstring permet de dclarer un titre de
% chapitre (ou section, sous-section) alternatif en texte seul, si besoin, qui
% est utilis par hyperref pour fabriquer un menu dans les fichiers compils

%Let $P$ be a polynomial in $n$ variables over $\K$ and denote $Y=(\K^*)^n\setminus V(P)$. The polynomial $P$ defines the following map:
%$$
%\begin{array}{cccc}
%\Phi: & Y & \longrightarrow & (\K^*)^{n+1}\\
%& z & \longmapsto & \left(z,P(z)\right).\\
%\end{array} 
%$$ 
%Put $S=Trop\left(\Phi(Y)\right)$.
%It follows from Kapranov's Theorem that $S$ is given by the tropical polynomial $``x_{n+1}+P_{trop}(x_1,\cdots ,x_n)"$.
%The tropical variety $S$ is called the \textit{tropical modification} of $\R^n$ along $P_{trop}$. Denote by $\pi_{n+1}:\R^{n+1}\rightarrow\R^n$ the projection forgetting the last coordinate. The map $\pi_{n+1}$ restrict to a surjective map $\pi_S:S\rightarrow\R^n$, one-to-one above $\R^n\setminus	 V\left(P_{trop}\right)$ and if $p\in V(P_{trop})$, then $\pi_S^{-1}(p)$ is a half ray, unbounded in the direction $(0,\cdots ,0,-1)$.
\begin{abstract}
We give a constructive proof using tropical modifications of the existence of a family of real algebraic plane curves with asymptotically maximal numbers of even ovals.
\end{abstract}
\section{Introduction}
Let $A$ be a non-singular real algebraic plane curve in $\CP^2$. Its real part, denoted by $\R A$, is a disjoint union of embedded circles in $\RP^2$. A component of $\R A$ is called an \textit{oval} if it divides $\RP^{2}$ in two connected components. If the degree of the curve $A$ is even, then the real part $\R A$ is a disjoint union of ovals. 
%Un ovale sépare $\RP^2$ en une composante connexe homéomorphe à un disque (l'\textit{intérieur} de l'ovale) et une composante connexe homéomorphe à un ruban de Moebius (l'\textit{extérieur} de l'ovale). 
An oval of $\R A$ is called \textit{even} (resp., \textit{odd}) if it is contained inside an even (resp., odd) number of ovals. 
%Remarquons que si $f$ est un polynôme homogène définissant $A$, alors, le degré de $f$ étant pair, le signe de $f$ en un point de $\RP^{2}$ est bien défini. Si l'on suppose que le signe de $f$ est négatif en dehors de tout ovale de $\R A$, alors les ovales pairs sont les ovales qui bordent extérieurement les composantes connexes de
%$$
%\RP^{2}_+=\lbrace x\in\RP^{2}\vert f(x)\geq 0\rbrace,
%$$  
%et les ovales impairs sont les ovales qui bordent extérieurement les composantes connexes de 
%$$
%\RP^{2}_-=\lbrace x\in\RP^{2}\vert f(x)\leq 0\rbrace.
%$$  
Denote by $p$ (resp., $n$) the number of even (resp., odd) ovals of $\R A$.
\\\\\textbf{Petrovsky inequalities:} For any real algebraic plane curves of degree $2k$, one has
$$
-\frac{3}{2}k(k-1)\leq p-n\leq \frac{3}{2}k(k-1)+1.
$$
%On peut parfois renforcer ces inégalités. Notons $p^{-}$ (resp., $n^{-}$) le nombre d'ovales pairs (resp., impairs) qui bordent extérieurement une composante connexe de $\RP_+^2$ (resp., $\RP_-^2$) de caractéristique d'Euler strictement négative.
%\\\\\textbf{Inégalités de Petrovsky renforcées:} Pour toute courbe algébrique réelle de degré $2k$ dans $\CP^2$, on a
%$$
%-\frac{3}{2}k(k-1)\leq p^{-}-n \:\: \mbox{ et } \:\: p-n^{-}\leq\frac{3}{2}k(k-1)+1.
%$$
%\\\\\textbf{Congruence de Rokhlin:} Pour toute $M$-courbe algébrique réelle de degré $2k$ dans $\CP^2$, on a
%$$
%p-n\equiv k^2 \mod 8.
%$$
%\\\\\textbf{Congruence de Gudkov-Krakhnov-Kharlamov:} Pour toute $(M-1)$-courbe algébrique réelle de degré $2k$ dans $\CP^2$, on a
%$$
%p-n\equiv k^2 \pm 1\mod 8.
%$$
One can deduce upper bounds for $p$ and $n$ from Petrovsky inequalities and Harnack theorem (which gives an upper bound for the number of components of a real algebraic curves with respect to its genus):
$$
p\leq \frac{7}{4}k^{2}-\frac{9}{4}k+\frac{3}{2},
$$
and
$$
n\leq \frac{7}{4}k^{2}-\frac{9}{4}k+1.
$$
In 1906, V. Ragsdale formulated the following conjecture.
\begin{conjecture}$($Ragsdale$)$
\\For any real algebraic plane curve of degree $2k$, one has
$$
p\leq \frac{3}{2}k(k-1)+1,
$$
and
$$
n\leq \frac{3}{2}k(k-1).
$$
\end{conjecture}
%In 1980, O. Viro  constructed examples of real algebraic plane curves of degree $2k$ with 
%$$
%n= \frac{3}{2}k(k-1)+1,
%$$
%and proposed to replace the second inequalities de la conjecture de Ragsdale par l'inégalité 
%$$
%n\leq \frac{3}{2}k(k-1)+1,
%$$
%voir \cite{Vir80}. Cette conjecture fut également proposée par Petrovsky (voir \cite{Petrovsky}).
In 1993, I. Itenberg used combinatorial patchworking to construct, for any $k\geq 5$, a real algebraic plane curve of degree $2k$ with
$$
p= \frac{3}{2}k(k-1)+1+\left\lfloor \dfrac{(k-3)^2+4}{8} \right\rfloor,
$$
and a real algebraic plane curve of degree $2k$ with
$$
n= \frac{3}{2}k(k-1)+\left\lfloor\dfrac{(k-3)^2+4}{8}\right\rfloor,
$$
see \cite{It93}. This construction was improved by B. Haas (see \cite{Haas95}) then by Itenberg (see \cite{It01}) and finally by E. Brugallé (see \cite{Brugalle2006}). Brugallé constructed a family of real algebraic plane curves with
$$
\lim_{k\rightarrow +\infty}\frac{p}{k^2}=\frac{7}{4},
$$
and a family of real algebraic plane curves with
$$
\lim_{k\rightarrow +\infty}\frac{n}{k^2}=\frac{7}{4}.
$$
In order to construct such families, Brugallé proved the existence of a family of real reducible curves $\mathcal{D}_n\cup\mathcal{C}_n$ in $\Sigma_n$, the $n$th Hirzebruch surface.
The curve $\mathcal{D}_n$ has Newton polytope 
$$
\Delta_n=Conv\left((0,0),(n,0),(0,1)\right),
$$ 
the curve $\mathcal{C}_n$ has Newton polytope 
$$
\Theta_n=Conv\left((0,0),(n,0),(0,2),(n,1)\right),
$$ 
and the chart of $\mathcal{D}_n\cup\mathcal{C}_n$ is homeomorphic to the one depicted in Figure \ref{courbebrugallered}.

\begin{figure}[h!]
\centerline{
\includegraphics[scale=0.8]{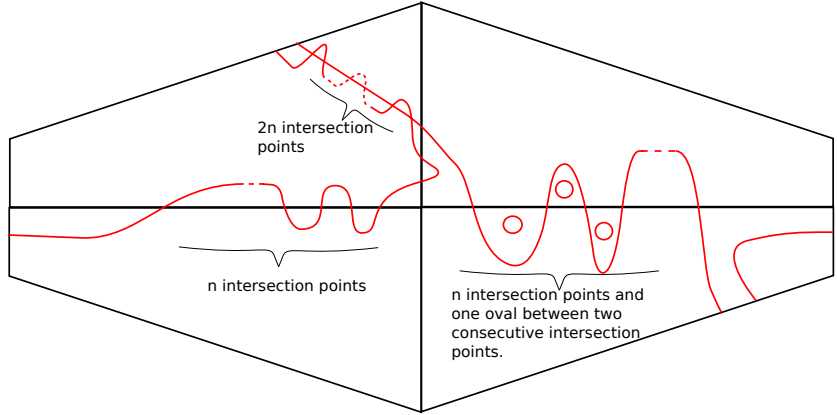}}
\setlength\abovecaptionskip{0cm}
\caption{The chart of $\mathcal{D}_n\cup\mathcal{C}_n$}
\label{courbebrugallered}
\end{figure}

%This family of real reducible curves was constructed by E. Brugall\'e in \cite{Brugalle2006} in order to produce real algebraic curves in $\CP^2$ with asymptotically maximal numbers of even ovals. Denote by $p$ the number of even ovals of a real algebraic curve of degree $2k$ in $\CP^2$. It follows from Petrovsky inequalities and Harnack Theorem that 
%$$
%p\leq \frac{7}{4}k^2-\frac{9}{4}k+\frac{3}{2}.
%$$
%In 1906, V. Ragsdale conjectured that
%$$
%p\leq \frac{3}{2}k(k-1)+1.
%$$
%In 1993, using Viro's combinatorial patchworking, I. Itenberg (see \cite{It93}) disproved Ragsdale's conjecture and constructed a family of real algebraic curves in $\CP^2$ of degree $2k$ with $\frac{13}{8}k^{2}+O(k)$ even ovals. This lower bound was successively improved by B. Haas (see \cite{Haas95}), by Itenberg (see \cite{It01}) and finally by Brugall\'e in \cite{Brugalle2006}.
Brugall\'{e}'s construction of this family of real reducible curves used so-called \textit{real rational graphs theoretical method}, based on Riemann existence theorem (see \cite{Brugalle2006} and \cite{Or03}). In particular, this method is not constructive. In this note, we give a constructive method to get such a family using tropical modifications and combinatorial patchworking for complete intersections (see Theorem \ref{patchtropcomp} and \cite{Sturmfels}). In Section \ref{amoebas}, we recall the notion of amoebas, the approximation of tropical hypersurfaces by amoebas and the notion of a real phase on a tropical hypersurface. In Section \ref{modif}, we remind the notion of a tropical modification along a rational function. In Section \ref{strategy} we give our strategy to construct the family $\mathcal{D}_n\cup\mathcal{C}_n$ and in Sections \ref{construction1} and \ref{construction2}, we explain the details of our construction.
\\
\\
\textbf{Acknowledgments.} I am very grateful to Erwan Brugallé and Ilia Itenberg for useful discussions and advisements.

\section{Amoebas and patchworking}
\label{amoebas}
In this section, we present a tropical formulation of the combinatorial patchworking theorem for nonsingular tropical hypersurfaces and complete intersections of nonsingular hypersurfaces. Amoebas appear as a fundamental link between classical algebraic geometry and tropical geometry.
\begin{definition}
Let $V\subset (\CC^{*})^{n}$ be an algebraic variety. Its amoeba (see \cite{GKZ}) is the set $\mathcal{A}=\mathrm{Log}(V)\subset\R^{n}$, where $\mathrm{Log}\left(z_1,\cdots ,z_n\right)=\left(\log \vert z_1\vert ,\cdots,\log \vert z_n \vert\right)$. Similarly, we may consider the map 
$$
\begin{array}{ccccc}
\mathrm{Log}_t & : & \left(\CC^*\right)^n & \to & \R^n\\
 & & (z_1\cdots,z_n) & \mapsto &  \left(\dfrac{\log \vert z_1\vert}{\log t},\cdots , \dfrac{\log \vert z_n\vert}{\log t}\right),
\end{array}
$$
for $t> 1$.
\end{definition}
\begin{definition}
Let $X$ and $Y$ be two non-empty compact subsets of a metric space $(M,\mathrm{d})$. Define their Hausdorff distance $\mathrm{d}_H(X,Y)$ by 
$$
\mathrm{d}_H(X,Y)=\max \left\lbrace \sup_{x\in X} \inf_{y\in Y}\mathrm{d}(x,y), \: \sup_{y\in Y} \inf_{x\in X}\mathrm{d}(x,y)\right\rbrace.
$$ 
\end{definition}
\begin{theorem}(Mikhalkin \cite{Mikpants}, Rullgard \cite{Rul01})
\label{thmepatchtrop}
\\Let $P(x)=``\sum_{i\in\Delta\cap\Z^n} a_ix^i"$ be a tropical polynomial in $n$ variables. Let $$f_t=\sum_{i\in\Delta\cap\Z^n} A_i(t)t^{a_i}z^i$$ be a family of complex polynomials and suppose that $A_i(t)\sim \gamma_i$ when $t$ goes to $+\infty$ with $\gamma_i\in\CC^{*}$. Denote by $Z(f_t)$ the zero-set of $f_t$ in $\left(\CC^*\right)^n$ and by $V(P)$ the tropical hypersurface associated to $P$. Then for any compact $K\subset\R^{n}$,
$$
\lim_{t\rightarrow + \infty}\mathrm{Log}_t(Z(f_t))\cap K=V(P)\cap K,
$$
with respect to the Hausdorff distance. We say that the family $Z(f_t)$ is an approximating family of the hypersurface $V(P)$.
\end{theorem}

\begin{remark}
Consider $A_i(t)\in\lbrace\pm 1 \rbrace$, for $i\in\Delta\cap\Z^n$. Then, the polynomial $f_t$ is a Viro polynomial (see for example \cite{It93} for the definition of a Viro polynomial).
\end{remark}
To give a tropical formulation of the combinatorial patchworking theorem for nonsingular tropical hypersurfaces, remind first the notion of a real phase for a nonsingular tropical hypersurface in $\R^{n}$. 
\begin{definition}
\label{realphasegeneral}
A real phase on a nonsingular tropical hypersurface $S$ in $\R^n$ is the data for every facet $F$ of $S$ of $2^{n-1}$ $n$-uplet of signs $\varphi_{F,i}=(\varphi_{F,i}^{1},\cdots ,\varphi_{F,i}^{n})$, $1\leq i\leq 2^{n-1}$ satisfying to the following properties:
\begin{enumerate}
\item If $1\leq i\leq 2^{n-1}$ and $v=(v_1,\cdots ,v_n)$ is an integer vector in the direction of $F$, then there exists $1\leq j\leq 2^{n-1}$ such that $(-1)^{v_k}\varphi_{F,i}^{k}=\varphi_{F,j}^{k}$, for $1\leq k\leq n$.
\item Let $H$ be a codimension $1$ face of $S$. Then for any facet $F$ adjacent to $H$ and any $1\leq i\leq 2^{n-1}$, there exists a unique face $G\neq F$ adjacent to $H$ and $1\leq j\leq 2^{n-1}$ such that $\varphi_{G,j}=\varphi_{F,i}$.
\end{enumerate}
A nonsingular tropical hypersurface equipped with a real phase is called a nonsingular real tropical hypersurface.
\end{definition}
\begin{example}
\label{extropline}
In Figure \ref{tropline}, we depicted a real tropical line.
\begin{figure}
 \begin{minipage}[l]{.46\linewidth}
  \centering\epsfig{figure=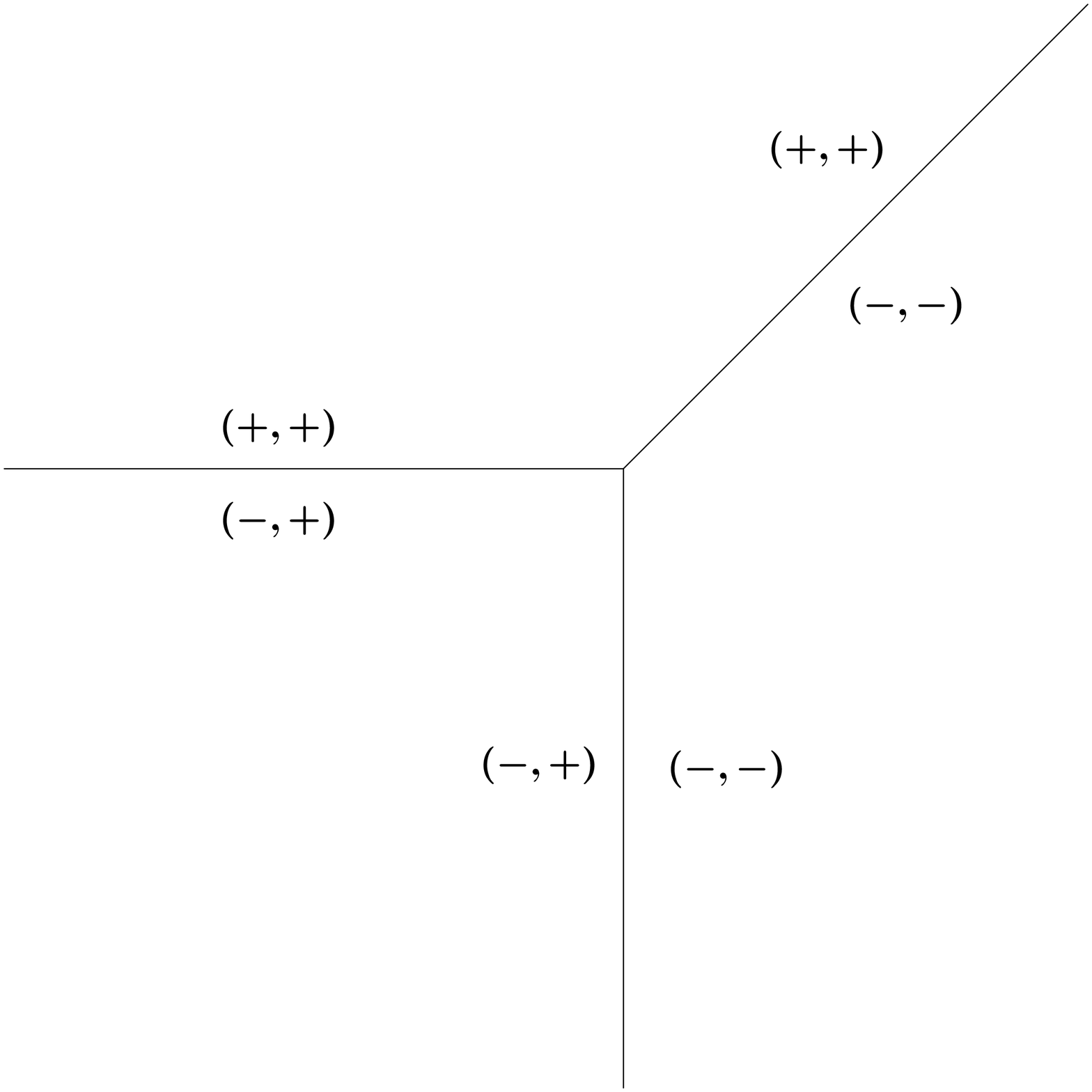,width=16cm}
  \caption{A real tropical line. \label{tropline}}
 \end{minipage} \hfill
 \begin{minipage}[l]{.46\linewidth}
  \centering\epsfig{figure=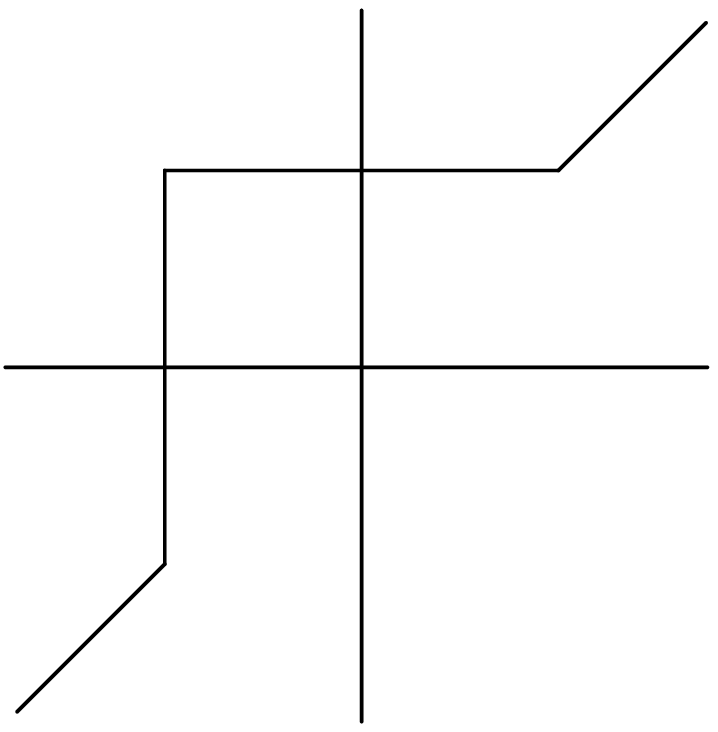,width=16cm}
  \caption{The real part of the real tropical line depicted in Figure \ref{tropline}. \label{realtropline}}
 \end{minipage}
\end{figure}
\end{example}

\begin{remark}
In the case of nonsingular tropical curves in $\R^2$, a real phase can also be described in terms of a ribbon structure (see \cite{BruItMiShaw}).
\end{remark}
\begin{definition}
 For any rational convex polyhedron $F$ in $\R^n$ defined by $N$ inequalities 
$$
<j_1,x>\leq c_1, \cdots , <j_N,x>\leq c_N,
$$ 
where $j_1\cdots j_N \in\Z^N$ and $c_1,\cdots, c_N\in\R$, denote by 
$F^{\exp}$ the rational convex polyhedron in $(\R^{*}_+)^{n}$ defined by the inequalities
$$
<j_1,x>\leq \exp(c_1), \cdots , <j_N,x>\leq \exp(c_N).
$$ 
Reciprocally for any rational convex polyhedron $F$ in $(\R^{*}_+)^{n}$ defined by the inequalities  
$$
<k_1,x>\leq d_1, \cdots , <k_N,x>\leq d_N,
$$ 
where $k_1\cdots k_N \in\Z^N$ and $d_1,\cdots, d_N\in\R^{*}_+$, denote by $F^{\log}$ the rational convex polyhedron in $\R^n$ defined by the inequalities
$$
<k_1,x>\leq \log(d_1), \cdots , <k_N,x>\leq \log(d_N).
$$
Extend these definitions to rational polyhedral complexes.
\end{definition}

For $\varepsilon=(\varepsilon_1,\cdots ,\varepsilon_n)\in(\Z/2\Z)^n$, denote by $s_\varepsilon$ the symmetry of $\R^n$ defined by
$$
s_\varepsilon(u_1,\cdots ,u_n)=((-1)^{\varepsilon_1}u_1,\cdots ,(-1)^{\varepsilon_n}u_n).
$$
Let $(S,\varphi)$ be a nonsingular real tropical hypersurface. Denote by $\mathcal{F}(S)$ the set of all facets of $S$.
\begin{definition}
The real part of $(S,\varphi)$ is
$$
\R S_\varphi=\bigcup_{F\in\mathcal{F}(S)}\bigcup_{1\leq i\leq 2^{n-1}} s_{\varphi_{F,i}}\left(F^{\exp}\right).
$$
\end{definition}
\begin{example}
In Figure \ref{realtropline}, we depicted the real part of the real tropical line from Example \ref{extropline}.

\end{example}
Let $S$ be a nonsingular tropical hypersurface in $\R^n$ given by a tropical polynomial $P$, and let $\varphi$ be a real structure on $S$. Denote by $\tau$ the dual subdivision of $S$. 
\begin{definition}
\label{compatibility}
A distribution of signs $\delta$ at the vertices of $\tau$ is called compatible with $\varphi$ if for any vertex $v$ of $\tau$, the following compatibility condition is satisfied. 
\begin{itemize}
\item For any vertex $w$ of $\tau$ adjacent to $v$, one has $\delta_v\neq\delta_w$ if and only if there exists $1\leq i\leq 2^{n-1}$ such that $\varphi_{F,i}=(+,\cdots ,+)$, where $F$ denotes the facet of $S$ dual to the edge connecting $v$ and $w$.
\end{itemize} 
\end{definition}
\begin{lemma}
\label{lemmarealphase}
For any real phase $\varphi$ on $S$, there exist exactly two distributions of signs at the vertices of $\tau$ compatible with $\varphi$. Reciprocally, given any distribution of signs $\delta$ at the vertices of $\tau$, there exists a unique real phase $\varphi$ on $S$ such that $\delta$ is compatible with $\varphi$.
\end{lemma}
\begin{proof}
Let $\varphi$ be a real phase on $S$. Choose an arbitrary vertex $v$ of $\tau$ and put an arbitrary sign $\varepsilon$ at $v$. Given a vertex of $\tau$ equipped with a sign, define a sign at all adjacent vertices by using the compatibility condition in Definition \ref{compatibility}. It gives a distribution of signs $\delta$ at the vertices of $\tau$ compatible with $\varphi$ such that $\delta_v=\varepsilon$. In fact, let $G$ be a face of $S$ of codimension $1$ and denote by $F_1,F_2,F_3$ the facets of $S$ adjacents to $G$. It follows from the definition of a real phase that either $\varphi_{F_i,k}\neq (+,\cdots ,+)$ for all $1\leq i\leq 3$ and $1\leq k \leq 2^{n-1}$, or that there exist exactly two indices $1\leq i,j\leq 3$ and such that 
$$
\varphi_{F_i,k_i}=\varphi_{F_j,k_j}=(+,\cdots ,+),
$$
where $k_i\in\lbrace 1, \cdots, 2^{n-1}\rbrace$ and $k_j\in\lbrace 1, \cdots, 2^{n-1}\rbrace$. This means exactly that going over any cycle $\Gamma$ made of edges of $\tau$, the signs at the vertices of $\Gamma$ change an even number of times, and the distribution of signs $\delta$ is well defined. The other distribution of signs at the vertices of $\tau$ compatible with $\varphi$ is the distribution $\delta'$ defined by $\delta'(v)=-\delta(v)$, for all vertices $v$ of $\tau$.
% Then the compatibility condition in Definition \ref{compatibility} determines in an unique way the sign at all vertices of $\tau$. To check that this gives a well defined distribution of signs at the vertices of $\tau$, consider $G$ be a face of $S$ of codimension $1$ and denote by $F_1,F_2,F_3$ the facets of $S$ adjacents to $G$. It follows from the definition of a real phase that either $\varphi_{F,i}\neq (+,\cdots ,+)$ for all $1\leq i\leq 3$, or that there exists exactly two indices $1\leq i,j\leq 3$ such that 
%$$
%\varphi_{F,i}=\varphi_{F,j}=(+,\cdots ,+).
%$$
%If $\Gamma$ denotes the cycle made of the dual edges of the faces $F_1,F_2,F_3$, 
\end{proof}
\begin{definition}
\label{Harnackphase}
Let $\Delta$ be a $2$-dimensional polytope in $\R_+^2$ and let $\tau$ be a primitive triangulation of $\Delta$. The \textit{Harnack distribution of signs} at the vertices of $\tau$ is defined as follows.
If $v$ is a vertex of $\tau$ with both coordinates even, put $\delta_v=-$, otherwise put $\delta_v=+$. The real phase compatible with $\delta$ is called the \textit{Harnack phase}. A $T$-curve associated to any primitive triangulation with a Harnack distribution of signs is a so-called simple Harnack curve (see for example \cite{It93} for the notion of a $T$-curve). Simple Harnack curves have some very particular properties (see \cite{Mik00}). 

\end{definition}

%\begin{definition}
%\label{realphase}
%Let $S$ be a nonsingular tropical hypersurface in $\R^n$ given by a tropical polynomial $P$. For any $(i_1,\cdots ,i_n)\in\Delta(P)\cap\Z^n$, attach a sign $\delta_{i_1,\cdots ,i_n}$ to the connected component of the complementary of $S$ dual to the point $(i_1,\cdots ,i_n)$. Such a distribution of signs is called a real phase on $S$. A nonsingular tropical hypersurface equipped with a real phase is called a nonsingular real tropical hypersurface.
%\end{definition}
%For $\varepsilon=(\varepsilon_1,\cdots ,\varepsilon_n)\in(\Z/2\Z)^n$, let $s_\varepsilon$ be the symmetry of $\R^n$ defined by
%$$
%s_\varepsilon(u_1,\cdots ,u_n)=((-1)^{\varepsilon_1}u_1,\cdots ,(-1)^{\varepsilon_n}u_n).
%$$
%Let $(S,\delta)$ be a nonsingular real tropical hypersurface. Let us think of the space $\R^{n}$ where $S$ sits as being the positive orthant $(\R_+^{*})^{n}$. Consider 
%$$
%\tilde{S}=\bigcup_{\varepsilon\in (\Z/2\Z)^n}s_\varepsilon(S),
%$$ the union of all the symmetric copies of $S$. Extend the distribution of signs to the other orthants using the following formula:
%$$
%\delta_{s_\varepsilon (i_1,\cdots ,i_n)}=\left(\prod_{j=1}^{n}(-1)^{\varepsilon_j i_j}\right)\delta_{i_1,\cdots ,i_n}.
%$$
%For every $\varepsilon\in (\Z/2\Z)^n$, delete the faces of $\tilde{S}\cap s_\varepsilon\left((\R_+^*)^n\right)$ separating equal signs. The result of this process, denoted by $\tilde{S}_\delta$, is called \textit{the real part of $(S,\delta)$}. 
For any nonsingular real tropical hypersurface $(S,\varphi)$ and any $\varepsilon\in(\Z/2\Z)^n$, put
$$
\R S_\varphi^\varepsilon=s_\varepsilon\left( \R S_\varphi \cap s_\varepsilon\left((\R_+^*)^n\right)\right).
$$
The set $\R S_\varphi^\varepsilon$ is a finite rational polyhedral complex in $(\R_+^*)^n$.
The following theorem is a corollary of Theorem \ref{thmepatchtrop}.
\begin{theorem}
\label{tropreal}
Let $S$ be a nonsingular tropical hypersurface given by the tropical polynomial $P(x)=``\sum_{i\in\Delta\cap\Z^{n}}a_ix^i"$, and let $\tau$ be the dual subdivision of $P$. Let $\varphi$ be a real phase on $S$ and let $\delta$ be a distribution of signs at the vertices of $\tau$ compatible to $\delta$. Put $f_t=\sum_{i\in\Delta\cap\Z^{n}}\delta_it^{a_i}z^i$. Then, for every $\varepsilon\in(\Z/2\Z)^n$ and for every compact $K\subset\R^n$, one has
$$
\lim_{t\rightarrow +\infty}\mathrm{Log}_t\left(Z(f_t)\cap s_\varepsilon\left((\R_+^*)^n\right)\right)\cap K=\left(\R S_\varphi^\varepsilon\right)^{\log} \cap K.
$$
We say that the family $Z(f_t)$ is an approximating family of $(S,\varphi)$.
\end{theorem}
\begin{definition}
\label{intersecttransvers}
Let $S_1\cdots ,S_k$ be $k$ tropical varieties in $\R^{n}$. We say that the varieties $S_i$, $1\leq i\leq k$, intersect transversely if every top-dimensional cell of $S_1\cap\cdots\cap S_k$ is a transverse intersection $\cap_{i=1}^{k} F_i$, where $F_i$ is a facet of $S_i$.
\end{definition}
The next theorem is a tropical reformulation of the combinatorial patchworking theorem for complete intersections (see \cite{Sturmfels}).
\begin{theorem}
\label{patchtropcomp}
Let $S_1,\cdots ,S_k$ be $k$ tropical hypersurfaces in $\R^n$ such that $S_j$ is given by the tropical polynomial $P^j(x)=``\sum_{i\in\Delta_j\cap\Z^{n}}a_i^jx^i"$, for $1\leq j\leq k$. Let $\tau^j$ be the dual subdivision of $P^j$, for $1\leq j\leq k$. Assume that the $S_1,\cdots ,S_k$ intersect transversely. Let $\varphi^{j}$ be a real phase on $S_j$ and $\delta^j$ be a distribution of signs at the vertices of $\tau^j$ compatible to $\delta^j$, for $1\leq j\leq k$. Put $f_t^j=\sum_{i\in\Delta_j\cap\Z^{n}}\delta_i^{j} t^{-a_i^j}z^i$. Then for every $\varepsilon\in(\Z/2\Z)^n$ and for every compact $K\subset \R^n$, one has
\begingroup\small
$$
\lim_{t\rightarrow +\infty}\mathrm{Log}_t\left(Z(f_t^1)\cap\cdots\cap Z(f_t^k)\cap s_\varepsilon\left((\R_+^*)^n\right)\right)\cap K=\left(\R S^{\varepsilon}_{1,\varphi^{1}}\right)^{\log}\cap \cdots\cap \left(\R S^{\varepsilon}_{k,\varphi^{k}}\right)^{\log}\cap K.
$$
\endgroup
We say that the family $\left(Z(f_t^1),\cdots,Z(f_t^k)\right)$ is an approximating family of 
$$
\left((S_1,\cdots,S_k),(\varphi^1,\cdots,\varphi^k)\right).
$$
\end{theorem}

%Let $S$ be a nonsingular tropical variety of dimension $k$ in $\R^n$ equipped with a real phase $\left(\varepsilon_{F,i}\right)$. As before, consider the space $\R^n$ where $S$ sits as being the positive orthant $(\R_+^*)^n$. Then the real part of $S$ relative to $\left(\varepsilon_{F,i}\right)$ is 
%$$
%S_\varepsilon=\bigcup_{F \mbox{\begin{tiny}
%facets of 
%\end{tiny}}S}\bigcup_{1\leq i\leq 2^{k}} s_{\varepsilon_{F,i}}\left(F\right).
%$$
%\begin{remark}
%Let $S$ be a tropical hypersurface in $\R^{n}$. Denote by $\Phi$ the set of real phases on $S$ in the sense of Definition \ref{realphase} and by $\Phi'$ the set of real phases on $S$ in the sense of Definition \ref{realphasegeneral}. Then there exists a bijection $\varphi:\Phi\rightarrow\Phi'$ such that for any $\delta\in\Phi$,
%$$
%\tilde{S_\delta}=S_{\varphi(\delta)}.
%$$ 
%\end{remark}

\section{Tropical modifications of $\R^{n}$}
\label{modif}
Tropical modifications were introduced by Mikhalkin in \cite{Mikhalkin}. We recall in this section the definition of a tropical modification of $\R^n$ along a rational function. More details can be found in \cite{Mikhalkin}, \cite{Bru_Medr}, \cite{Shawthese} and \cite{BruItMiShaw}. 
Let $f:\R^{n}\rightarrow \R$ and $g:\R^{n}\rightarrow \R$ be two tropical polynomials. One may consider the rational tropical function $h=``\dfrac{f}{g}"$. Denote by $V(f)$ the tropical hypersurface associated to $f$ and by $V(g)$ the tropical hypersurface associated to $g$.
\begin{definition}
The tropical modification of $\R^{n}$ along $h$, denoted by $\R^n_h$, is the tropical hypersurface of $\R^{n+1}$ defined by 
$``x_{n+1}g(x)+f(x)"$.
\end{definition}
We may also describe $\R^n_h$ in a more geometrical way. Consider the graph $\Gamma_h$ of the piecewise linear function $h$. It is a polyhedral complex in $\R^{n+1}$. Equip the graph $\Gamma_h$ with the constant weight function equal to $1$ on it facets. In general, this graph is not a tropical hypersurface of $\R^{n+1}$ as it is not balanced at faces $F$ of codimension one. At every codimension one face $F$ of $\Gamma_h$ which fails to satisfy the balancing condition, add a new facet as follows. Denote by $w$ the integer number such that the balancing condition is satisfied at $F$ if we attach to $F$ a facet $F^{-1}$ in the $(0,\cdots,0,-1)$-direction equipped with the weight $w$. If $w>0$, attach $F^{-1}$ (equipped with $w$) to $F$ and if $w<0$, attach to $F$ a facet $F^{+1}$ in the $(0,\cdots,0,1)$-direction equipped with the weight $-w$. If $V(f)$ and $V(g)$ intersect transversely, then the tropical modification $\R^n_h$ is obtained by attaching to the graph $\Gamma_h$ the intervals $$\left](x,-\infty),(x,h(x))\right]$$ for all $x$ in the hypersurface $V(f)$, and $$\left[(x,h(x)),(x,+\infty)\right[$$ for all $x$ in the hypersurface $V(g)$ and by equipping each new facet with the unique weight so that the balancing condition is satisfied. 

\begin{definition}
The principal contraction
$$
\delta_h:\R^{n}_h\rightarrow \R^{n}
$$
associated to $h$ is the projection of $\R^{n}_h$ onto $\R^{n}$.
\end{definition}
The principal contraction $\delta_h$ is one-to-one over $\R^n\setminus \left(V(f)\cup V(g)\right)$. 
\begin{example}
In Figure \ref{tropplane}, we depicted the tropical modification of $\R^2$ along the tropical line given by the tropical polynomial $``x+y+0"$. It is the tropical plane in $\R^3$ given by the tropical polynomial $``x+y+z+0"$.
\begin{figure}[h!]
\centerline{
\includegraphics[scale=0.05]{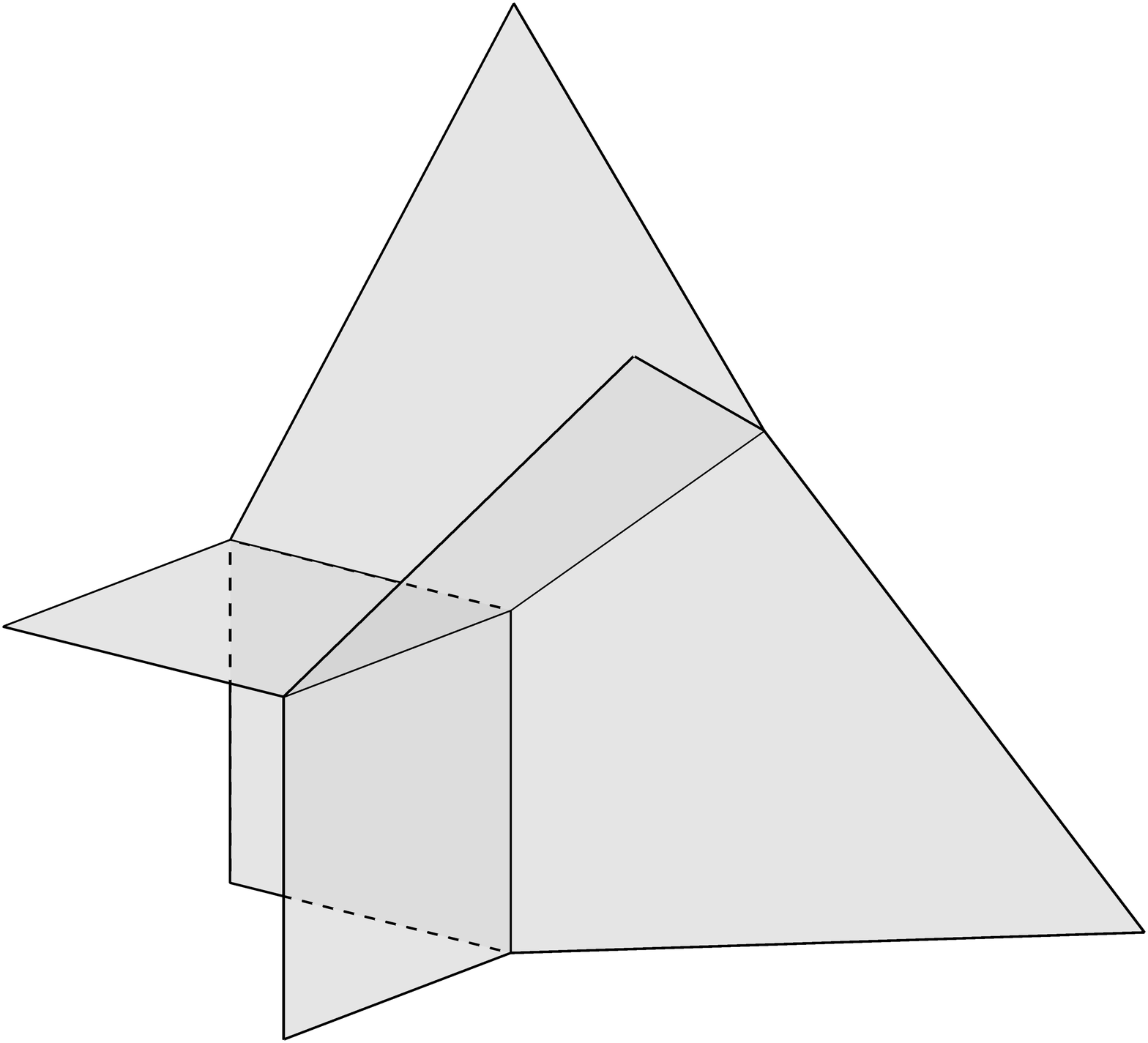}}
\setlength\abovecaptionskip{0cm}
\caption{The tropical modificiation of $\R^2$ along the tropical line $``x+y+0"$.}
\label{tropplane}
\end{figure} 
In Figure \ref{tropmodif}, we depicted the tropical modification of $\R^2$ along $``\frac{P}{Q}"$, where $``P=x+y+0"$ and $``Q=y+(-1)"$. The tropical curves $V(P)$ and $V(Q)$ intersect transversely. In Figure \ref{tropmodif2}, we depicted the tropical modification of $\R^2$ along $``\frac{P}{Q_1}"$, where $``Q_1=y+0"$. The tropical curves $V(P)$ and $V(Q)$ do not intersect transversely.
\begin{figure}[h!]
\centerline{
\includegraphics[scale=0.16]{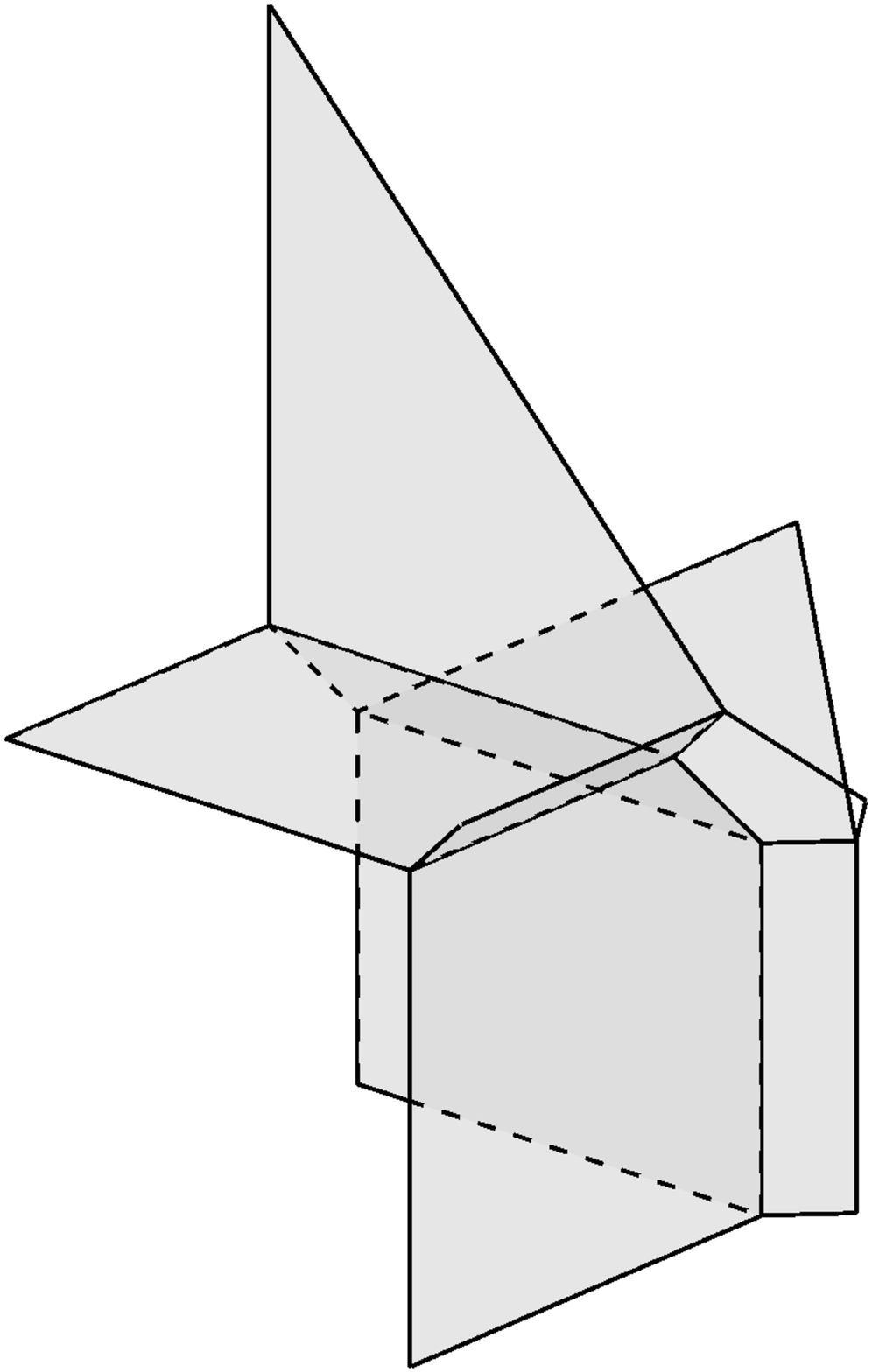}}
\setlength\abovecaptionskip{0cm}
\caption{The tropical modificiation of $\R^2$ along $``\frac{P}{Q}"$, where $``P=x+y+0"$ and $``Q=y+(-1)"$.}
\label{tropmodif}
\end{figure} 
\begin{figure}[h!]
\centerline{
\includegraphics[scale=0.1]{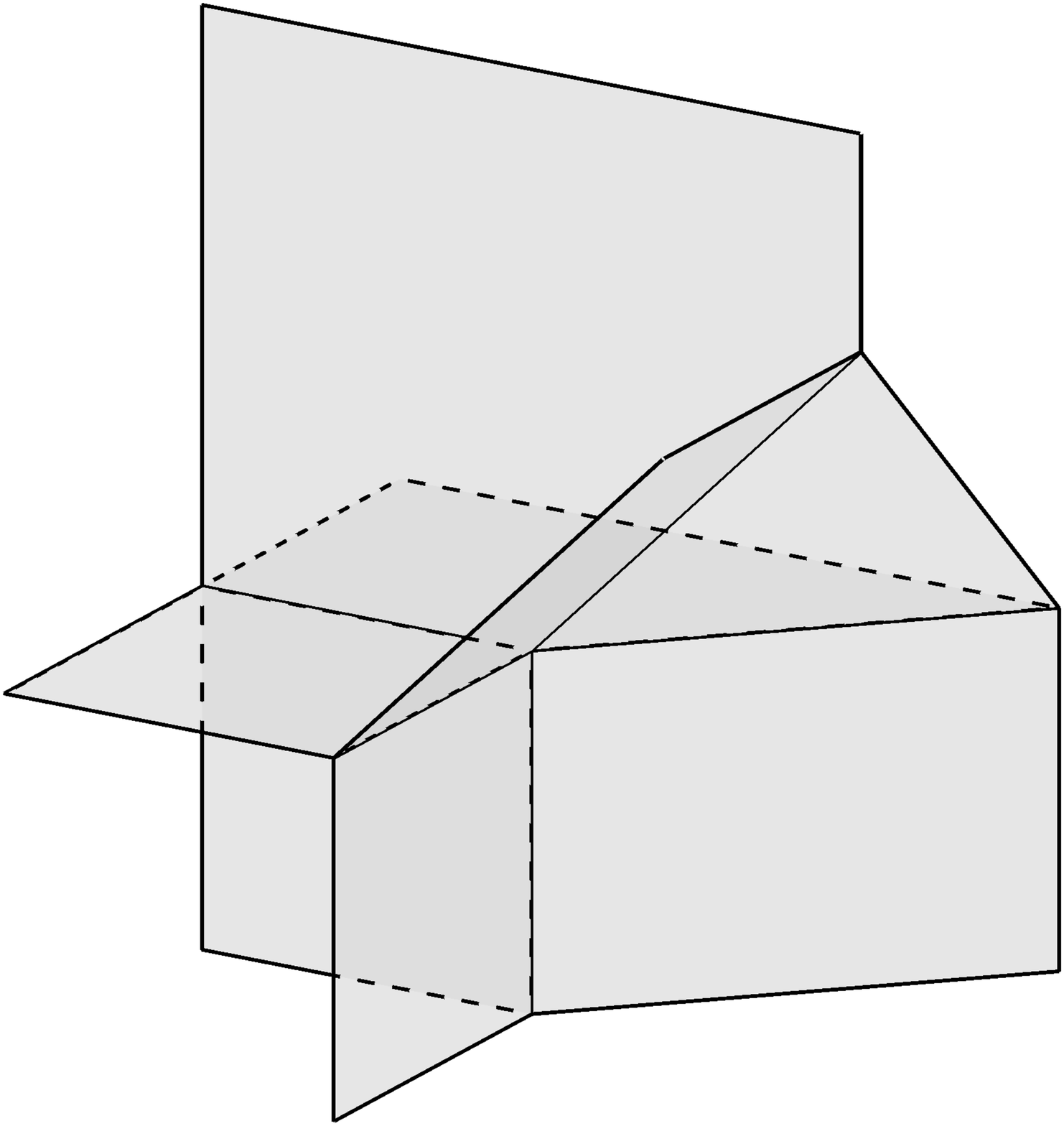}}
\setlength\abovecaptionskip{0cm}
\caption{The tropical modificiation of $\R^2$ along $``\frac{P}{Q_1}"$, where $``P=x+y+0"$ and $``Q_1=y+0"$.}
\label{tropmodif2}
\end{figure} 
\end{example}

\section{Strategy of the construction}
\label{strategy}
Let $n\geq 1$. We construct the curve $\mathcal{D}_n$ (resp., $\mathcal{C}_n$) in a $1$-parameter family of curves $\mathcal{D}_{n,t}$ (resp., $\mathcal{C}_{n,t}$). To construct such families of curves, we construct a tropical curve $D_n$ with Newton polytope $\Delta_n$ (see Figure \ref{courbe1} for the case $n=3$) and a tropical curve $C_n$ with Newton polytope $\Theta_n$ (see Figure \ref{courbe2} for the case $n=3$). The family of curves $\mathcal{D}_{n,t}$ (resp. $\mathcal{C}_{n,t}$) then appears as an approximating family of the tropical curve $D_n$ (resp., $C_n$). It turns out that the tropical curves $C_n$ and $D_n$ do not intersect transversely (see Figure \ref{courbe3}), so we can not use directly combinatorial patchworking to determine the mutual position of the  curves $\mathcal{D}_{n,t}$ and $\mathcal{C}_{n,t}$. We consider then the tropical modification $X_n$ of $\R^2$ along $P_n$, where $P_n$ is a tropical polynomial defining $D_n$. In this new model, the curve $D_n$ is the boundary in the vertical direction of the compactification of $X_n$ in $\T^n$, and if $\tilde{C}_n$ is a lifting of $C_n$ in $X_n$ (see Definition \ref{deflift}), then the compactification of $\tilde{C}_n$ in $\T^n$ intersects $D_n$ transversely. Then, we show that the curve $\tilde{C}_n$ is the transverse intersection of $X_n$ with $Y_n$, a tropical modification of $\R^2$ along a tropical rational function (see Definition \ref{intersecttransvers}). We define real phases $\varphi_{D_n}$ on $D_n$ and real phases $\varphi_{C_n}$ on $C_n$ (see  Figure \ref{courbe7} and Figure \ref{courbe8} for the case $n=3$) and we construct real phases $\varphi_{X_n}$ on $X_n$ and real phases $\varphi_{Y_n}$ on $Y_n$ satisfying compatibility conditions with $\varphi_{D_n}$ and $\varphi_{C_n}$ (see Lemma \ref{lemmarealphases}). 
%$\varphi_{C_n}$ on $C_n$ and $\varphi_{\tilde{C}_n}$ on $\tilde{C}_n$ such that the following compatibility conditions are satisfied:
%\begin{enumerate}
%\item $\R D_n$ is the projection of the union of all vertical faces of $\R X_n$.
%\item $\pi^\R\left(\R\tilde{C}_n\right)=\R C_n$, where $\pi^\R:\R^3\rightarrow\R^2$ denotes the projection forgetting the last factor.
%\item $\R X_n\cap \R Y_n=\R \tilde{C}_n$.
%\end{enumerate}  
It follows from Theorem \ref{tropreal} that there exists a family of real polynomials $P_{n,t}$ with Newton polytopes $\Delta_n$ such that if we put
$$
\mathcal{D}_{n,t}=\left\lbrace P_{n,t}(x,y)=0\right\rbrace,
$$ 
and
$$
\mathcal{X}_{n,t}=\left\lbrace z+P_{n,t}(x,y)=0\right\rbrace,
$$
one has
$$
\lim_{t\rightarrow +\infty}\mathrm{Log}_t \left(\R \mathcal{D}_{n,t}\cap s_\varepsilon\left((\R_+^*)^2\right)\right)\cap V=\left(\R D_n^\varepsilon\right)^{\log}\cap V,
$$
and
$$
\lim_{t\rightarrow +\infty}\mathrm{Log}_t \left(\R \mathcal{X}_{n,t}\cap s_\eta\left((\R_+^*)^3\right)\right)\cap W=\left(\R X_n^\eta\right)^{\log}\cap W,
$$
for any $\varepsilon \in (\Z/2\Z)^2$, any $\eta \in (\Z/2\Z)^3$, any compact $V\subset \R^2$ and any compact $W\subset \R^3$.
It follows from Theorem \ref{patchtropcomp} that there exists a family of surfaces $\mathcal{Y}_{n,t}$ such that for any $\varepsilon\in (\Z/2\Z)^3$ and any compact $V\subset \R^3 $, one has
$$
\lim_{t\rightarrow +\infty}\mathrm{Log}_t\left(\R \mathcal{X}_{n,t}\cap \R \mathcal{Y}_{n,t}\cap s_\varepsilon\left((\R_+^*)^3\right)\right)\cap V =\left( \R X_n^\varepsilon\right)^{\log}\cap \left( \R Y_n^\varepsilon\right)^{\log}\cap V.
$$
Consider the projection $\pi^\C:\left(\C^*\right)^3\rightarrow (\C^*)^2$ forgetting the last coordinate. For every $t$, put
$$
\mathcal{C}_{n,t}=\pi^\C(\mathcal{X}_{n,t}\cap\mathcal{Y}_{n,t}).
$$
Then, the Newton polytope of $\mathcal{C}_{n,t}$ is $\Theta_n$ and we show that for $t$ large enough, the chart of $\mathcal{D}_{n,t}\cup \mathcal{C}_{n,t}$ is homeomorphic to the chart depicted in Figure \ref{courbebrugallered}. Thus, we put
$$
\mathcal{D}_n=\mathcal{D}_{n,t},
$$
and
$$
\mathcal{C}_n=\mathcal{C}_{n,t}
$$
for $t$ large enough.
%$$
%\mathcal{C}^1_{n,t}=\left\lbrace P_{n_t}=0\right\rbrace.
%$$ 
%This means that for $t$ big enough,
%$$
%\left(\R\mathcal{S}_{n,t},\R\mathcal{C}^2_{n,t}\right)\simeq\left(\R (X_n)_{\varphi_{P_n}},\R(\tilde{C}_n)_{\varphi_{\tilde{C}_n}}\right).
%$$
%The surface $\mathcal{S}_{n,t}$ is given by the zero set of a polynomial $z+P_{n,t}(x,y)$. Put 
%$$
%\mathcal{C}^1_{n,t}=\left\lbrace P_{n_t}=0\right\rbrace.
%$$ 
%Then one has
%$$
%\R\mathcal{S}_{n,t}=(\R^*)^2\setminus\R\mathcal{C}^1_{n,t},
%$$
%and
%$$
%\left((\R^*)^2\setminus\R\mathcal{C}^1_{n,t},\R\mathcal{C}^2_{n,t}\right)\simeq\left(\R (X_n)_{\varphi_{P_n}},\R(\tilde{C}_n)_{\varphi_{\tilde{C}_n}}\right).
%$$
%Denote by $\pi:(\C^*)^3\rightarrow (\C^*)^2$ the projection forgetting the last factor. The curve $\left(C_n^2=0\right)$ is given by $\pi(\mathcal{C}^n_t)$ for big enough $t$. 

%\section{Lifting of marked tropical curves in tropical modifications of \texorpdfstring{$\R^2$}{TEXT}}
%We give in this section some important definitions for our construction.

%If $h=``\dfrac{f}{g}"$, then it follows from the definition that $x_1\cdots x_k\in C\cap V(f)$, where $V(f)$ denotes the tropical curve associated to $f$. 
%\\Let $h$ be a tropical rational function on $\R^2$ such that the modification $\R^2_h$ is nonsingular and let $C\subset\R^2$ be a tropical curve with $k$ marked points $x_1,\cdots ,x_k$ such that 
%$\left( C,x_1\cdots ,x_k\right)$ can be lifted to $\R^2_h$. Denote 
%by $\tilde{C}$ a lifting  

\section{Construction of the surfaces $X_{n}$ and $Y_n$, and of the curve $\tilde{C}_{n}$}
\label{construction1}
 Consider the subdivision of $\Delta_n$ given by the triangles 
$$
\Delta_n^k=Conv\left((k,0),(0,1),(k+1,0)\right),
$$ for $0\leq k\leq n-1$. Consider the subdivision of $\Theta_n$ given by the triangles 
\begin{itemize}
\item $K_n^k=Conv\left((k,0),(k,1),(k+1,0)\right)$,
\item $L_n^k=Conv\left((k,1),(k+1,0),(k+1,1)\right)$ and
\item $M_n^k=Conv\left((k,1),(k,2),(k+1,1)\right)$, 
\end{itemize}
for $0\leq k\leq n-1$.
Consider a tropical curve $D_{n}$ dual to the subdivision $\left(\Delta_n^k\right)_{0\leq k\leq n-1}$ of $\Delta_n$ (see Figure \ref{courbe1} for the case $n=3$), a tropical curve $C_{n}$ dual to the subdivision $\left(K_n^k,L_n^k,M_n^k\right)_{0\leq k\leq n-1}$ of $\Theta_n$ (see Figure \ref{courbe2} for the case $n=3$), and $2n$ marked points $x_1,\cdots ,x_{2n}$ on $C_n$, such that $D_n$, $C_n$ and $x_1,\cdots ,x_{2n}$ satisfy the following conditions.

\begin{figure}
 \begin{minipage}[l]{.46\linewidth}
  \centering\epsfig{figure=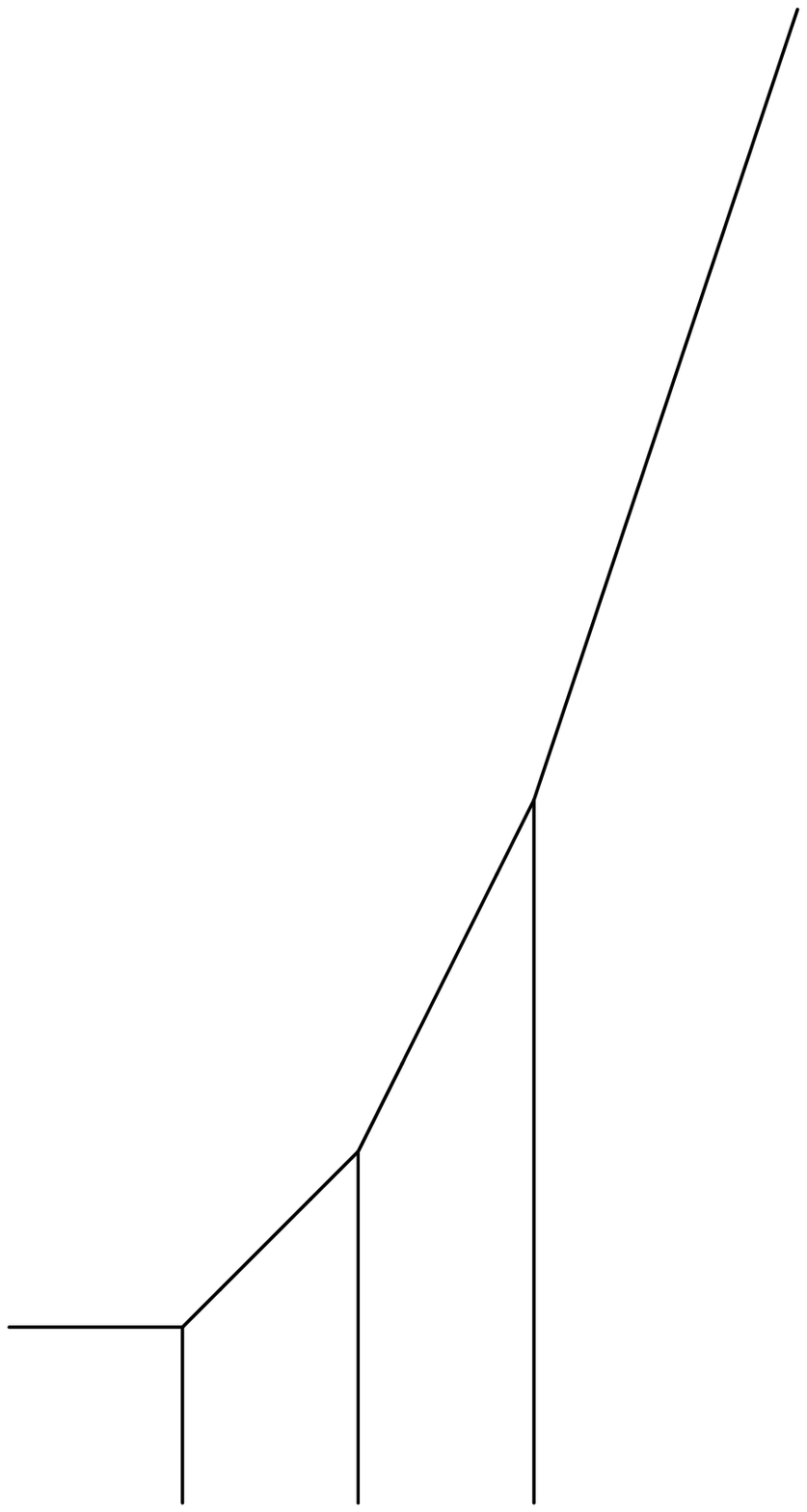,width=17cm}
  \caption{Tropical curve $D_3$. \label{courbe1}}
 \end{minipage} \hfill
 \begin{minipage}[l]{.46\linewidth}
  \centering\epsfig{figure=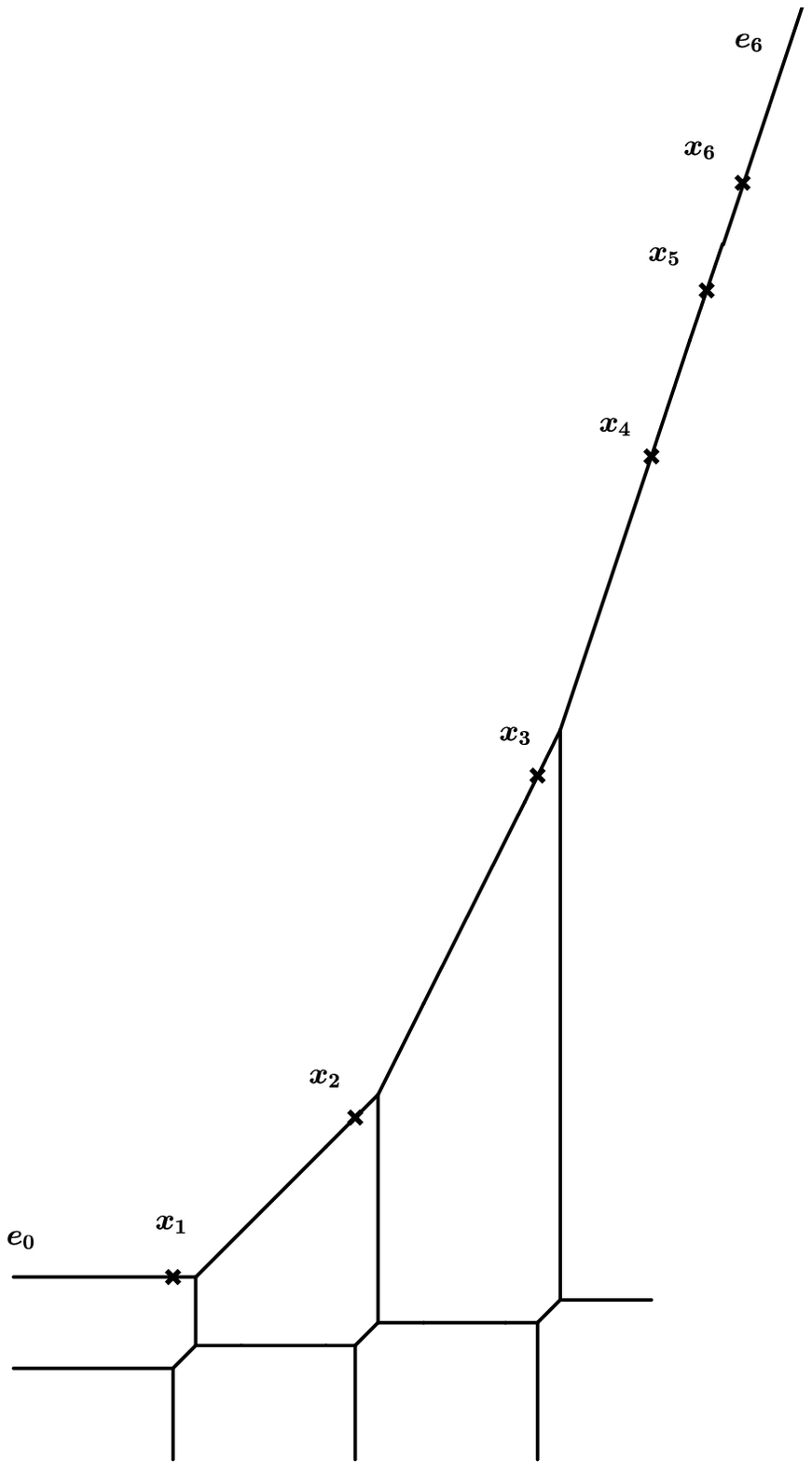,width=17cm}
  \caption{Tropical curve $C_3$. \label{courbe2}}
 \end{minipage}
\end{figure}

\begin{enumerate}
\item For any $0\leq k\leq n-1$, the coordinates of the vertex of $D_n$ dual to $\Delta_n^k$ are equal to the coordinates of the vertex of $C_n$ dual to $M_n^k$.
\item For any $1\leq k\leq n$ the first coordinate of $x_{k}$ is equal to the first coordinate of the vertex dual to $K_n^{k-1}$.
\item For any $n+1\leq k \leq 2n$, the marked point $x_k$ is on the edge of $C_n$ dual to the edge $\left[(0,2),(n,1)\right]$ of $\Theta_n$.
\end{enumerate}
For each marked point $x_i$, $1\leq i\leq 2n$, refine the edge of $C_n$ containing $x_i$ by considering the marked point $x_i$ as a vertex of $C_n$. In Figure \ref{courbe3}, we draw the tropical curves $D_3$ and $C_3$ on the same picture. 
\begin{figure}[h!]
\centerline{
\includegraphics[width=17cm]{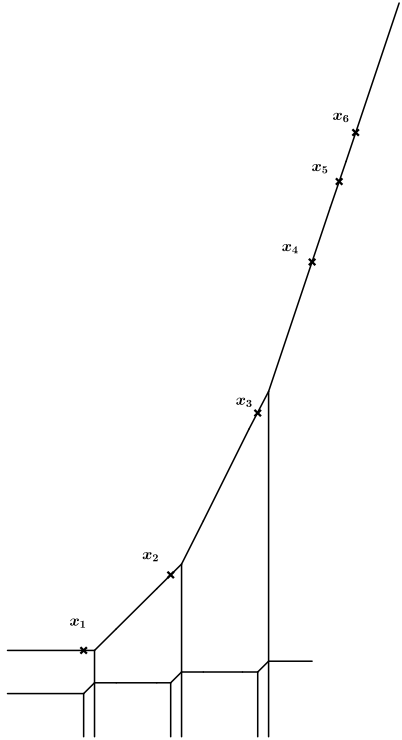}}
\setlength\abovecaptionskip{0cm}
\caption{The curves $D_3$ and $C_3$.}
\label{courbe3}
\end{figure}
 Denote by $P_n$ a tropical polynomial defining the tropical curve $D_n$ and put $X_n=\R^2_{P_n}$.
\begin{definition} 
\label{deflift0}
Let $C\subset\R^2$ be a tropical curve with $k$ vertices of valence $2$ denoted by $x_1,\cdots ,x_k$ $($called the marked points of $C)$. We say that a tropical curve $\tilde{C}\subset\R^3$ is a lift of $\left( C,x_1,\cdots ,x_k\right)$ if the following conditions are satisfied.
\begin{itemize} 
\item $\pi^\R\left(\tilde{C}\right)=C$, where $\pi^\R:\R^3\rightarrow\R^2$ denotes the vertical projection on the first two coordinates.
\item Any infinite vertical edge of $\tilde{C}$ is of the form $\left](x,-\infty),(x,r)\right]$, with $x\in\R^{2}$ and $r\in\R$.
\item  An edge $e$ of $\tilde{C}$ is an infinite vertical edge if and only if $\pi^\R (e)\subset\lbrace x_1, \cdots,x_k \rbrace$.
\item For any point $x$ in the interior of an edge $e$ of $C$ such that $(\pi^\R)^{-1}(x)\cap \tilde{C}$ is finite, one has
$$
w(e)=\sum_{i=1}^l w(f_i)\left[\Lambda_e:\Lambda_{f_i}\right],
$$
where $f_1,\cdots, f_l$ are the edges of $\tilde{C}$ containing the preimages of $x$, the weight of $e$ $($resp., $f_i)$ is denoted by $w(e)$ $($resp., $w(f_i))$ and $\Lambda_e$ $($resp., $\Lambda_{f_i})$ denotes the sublattice of $\Z^3$ generated by a primitive vector in the direction of $e$ $($resp., $f_i)$ and by the vector $(0,0,1)$.
\end{itemize} 
\end{definition}
\begin{remark}
\label{remarklift}
Let $C$ be a nonsingular tropical curve in $\R^2$ with $k$ marked points $x_1,\cdots, x_k$ such that there exists a lifting $\tilde{C}$ of $\left(C,x_1,\cdots,x_k\right)$ in $\R^3$. Then, it follows from Definition \ref{deflift0} that for any edge $e$ of $C$, there exists a unique edge $f$ of $\tilde{C}$ such that $\pi^\R(f)=e$. Moreover, one has $w(f)=1$. It follows from the balancing condition that at any trivalent vertex $v$ of $C$, the directions of the lifts of any two edges adjacents to $v$ determine the direction of the lift of the third edge adjacent to $v$. At any marked point $x_i$, the direction of the lift of an edge adjacent to $x_i$ and the weight of the infinite vertical edge of $\tilde{C}$ associated to $x_i$ determine the direction of the lift of the other edge adjacent to $x_i$.  

\end{remark}
\begin{definition}
\label{deflift}
Let $h$ be a tropical rational function on $\R^2$ and let $\R^2_h$ be the tropical modification of $\R^2$ along $h$.
Let $C\subset\R^2$ be a tropical curve with $k$ marked points $x_1,\cdots ,x_k$. We say that the marked tropical curve $\left( C,x_1,\cdots ,x_k\right)$ can be lifted to $\R^2_h$ if there exists a lifting $\tilde{C}$ of $\left( C,x_1,\cdots ,x_k\right)$ in $\R^3$ such that $\tilde{C}\subset\R^2_h$.
\end{definition} 
\begin{remark}
\label{remarkliftmod}
Assume that a trivalent marked tropical curve $\left(C,x_1,\cdots,x_{k}\right)$ can be lifted to some tropical modification $\R^2_h$, where $h=``\dfrac{f}{g}"$ is some tropical rational function. It follows from the definition of a tropical modification that outside of $V(f)\cup V(g)$, the lift of an edge of $C$ to $\R^2_h$ is uniquely determined. 
\end{remark}
\begin{lemma}
There exists a unique lifting of the marked tropical curve \linebreak $\left(C_n,x_1,\cdots,x_{2n}\right)$ to $X_n$. 
\end{lemma}
\begin{proof}
It follows from Remark \ref{remarkliftmod} that for all edges $e$ of $C_n$ not belonging to $D_n$, the lift of $e$ in $X_n$ is uniquely determined. The lift of a marked point $x_i$ is an edge $s_i$ of the form $\left](x_i,-\infty),(x_i,r_i)\right]$, where $r_i\in\R$. Denote by $e_0$ the edge of $C_n$ dual to the edge $\left[ (0,1),(0,2)\right]$ and by $e_{2n}$ the infinite edge of $C_n$ dual to the edge $\left[ (0,2),(n,1) \right]$ adjacent to $x_{2n}$ (see Figure \ref{courbe2} for the case $n=3$). 
%Following Remark \ref{remarklift}, one can see that in this case the direction of a lift $\tilde{e}_0$ of $e_0$ determine the direction of a lift of any edge of $C_n$. 
One can see from Remark \ref{remarklift} that if the direction of $\tilde{e}_0$ is $(1,0,s)$ then the direction of the lift of $e_{2n}$ is $(1,0,s+\sum w(s_i))$. Since the edges $e_0$ and $e_{2n}$ are unbounded, one has $s\geq 0$ and $s+\sum w(s_i)\leq n$. Thus, the direction of $\tilde{e}_0$ is $(1,0,0)$ and $w(s_i)=1$, for $1\leq i\leq n$. One can see following Remark \ref{remarklift} that in this case the direction of a lift of any edge of $C_n$ is uniquely determined. The only potential obstructions on the lifts of the edges of $C_n$ to close up come from the cycles of $C_n$. Denote by $Z_k$ the cycle bounding the face dual to the vertex $(k-1,1)$ of $\Theta_n$, for $2\leq k\leq n$ (see Figure \ref{cycle}). Denote by $e^k_1,\cdots ,e^k_6$ the edges of $Z_k$ as indicated in Figure \ref{cycle}. Orient the edges $e^{k}_i$ with the orientation coming from the counterclockwise orientation of the cycle $Z_k$.
\begin{figure}[h!]
\centerline{
\includegraphics[width=20cm]{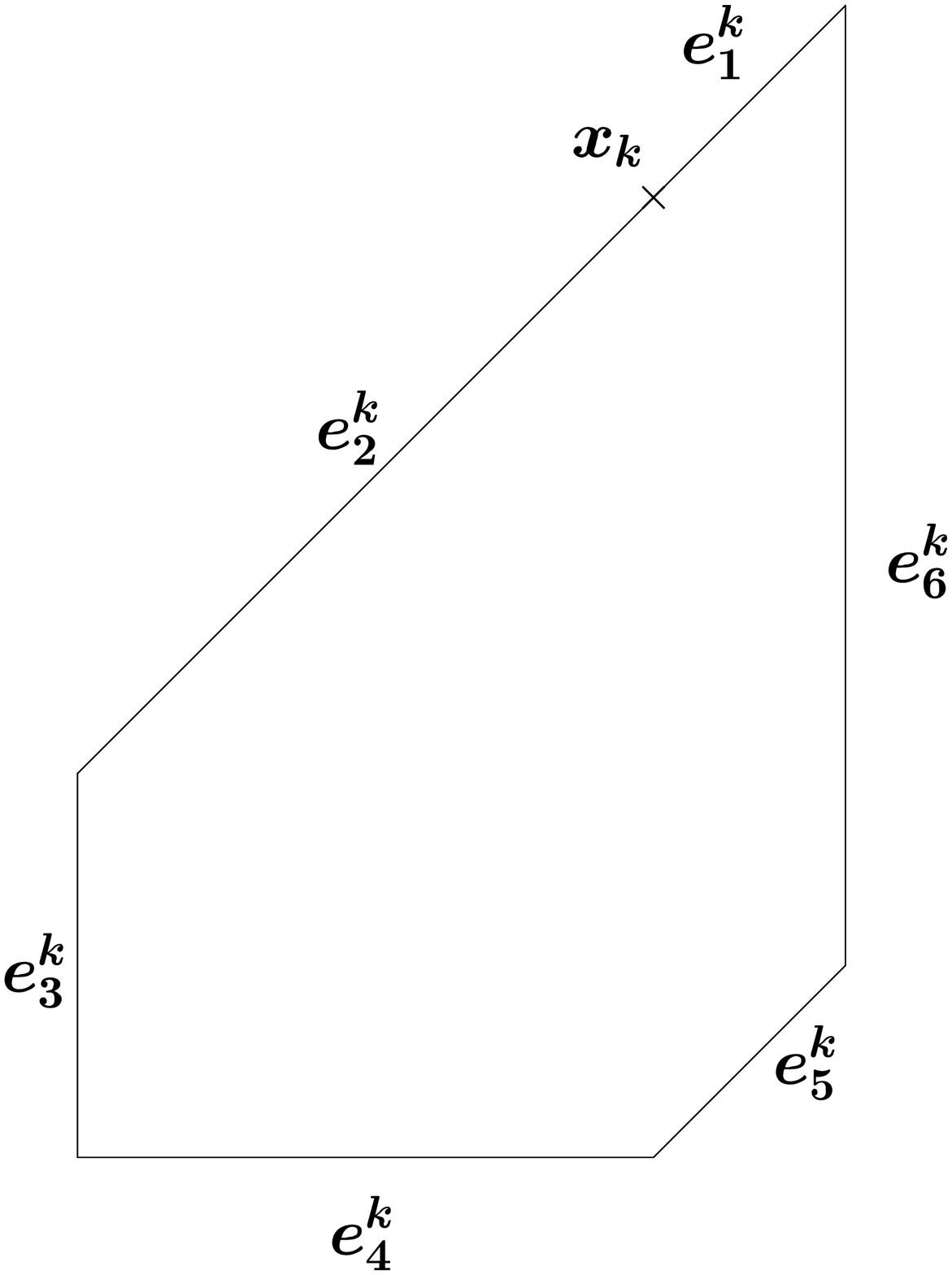}}
\caption{The cycle $Z_k$.}
\label{cycle}
\end{figure}
Denote by $\left[e_i^k\right]$ the lift of the edge $e_i^k$, oriented with the orientation coming from the one of $e_i^k$. Denote by $\overrightarrow{e^k_i}$ the primitive vector in the direction of $\left[e^k_i\right]$, and denote by $l_i^k$ the integer length of $e^k_i$. The lifts of the edges of $Z_k$ close up if and only if the following equation is satisfied:
$$
\sum_{i=1}^{i=6}l_i^k\overrightarrow{e^k_i}=0.
$$
This equation is equivalent to
$$
l_1^k
\left( \begin{array}{c}
-1 \\
-k \\
-1 \\
\end{array} \right)
+l_2^{k}
\left( \begin{array}{c}
-1 \\
-k \\
0 \\
\end{array} \right)
+l_3^{k}
\left( \begin{array}{c}
0 \\
-1 \\
1 \\
\end{array} \right)
+
l_4^k
\left( \begin{array}{c}
1 \\
0 \\
k \\
\end{array} \right)
+l_5^k
\left( \begin{array}{c}
1 \\
1 \\
k \\
\end{array} \right)
+l_6^k
\left( \begin{array}{c}
0 \\
1 \\
-1 \\
\end{array} \right)
=0.
$$   
This equation is equivalent on the cycle $Z_k$ to
$$
\left\lbrace \begin{array}{c}
l_2^k=l_4^k, \\
l_1^k=l_5^k. \\
\end{array} \right.
$$
This is equivalent to say that the first coordinate of the marked point $x_{k}$ is equal to the first coordinate of the vertex dual to $K_n^{k-1}$, for $2\leq k\leq n$.
\end{proof}
We construct a tropical rational function $h_n$ such that the curve $\tilde{C}_n$ is the transverse intersection of $X_n$ and $\R^2_{h_n}$.
Consider the subdivision of $\Theta_n$ given by the triangles 
\begin{itemize}
\item $G_n^k=Conv\left((k,0),(0,1),(k+1,0)\right)$,
\item $H_n^k=Conv\left((k,1),(k+1,1),(n,0)\right)$ and
\item $I_n^k=Conv\left((k,1),(0,2),(k+1,1)\right)$, 
\end{itemize}
for $0\leq k\leq n-1$. Consider a tropical curve $F_n$ dual to the subdivision \linebreak $\left(G_n^k,H_n^k,I_n^k\right)_{0\leq k\leq n-1}$ and a horizontal line $E$ such that the following conditions are satisfied (see Figure \ref{courbe4} and Figure \ref{courbe6} for the case $n=3$).
\begin{enumerate}
\item The  horizontal line $E$ is below any vertex of $C_n$.
\item The marked point $x_k$ is on the edge dual to the edge $[(k-1,0),(k,0)]$, for $1\leq k\leq n$.
\item The edge of $F_n$ dual to the edge $[(1,1),(n,0)]$ intersects transversely the edge of $C_n$ dual to the edge $[(0,2),(n,1)]$ at the point $x_{n+1}$.
\item The edge of $F_n$ dual to the edge $[(k,1),(k+1,1)]$ intersects transversely the edge of $C_n$ dual to the edge $[(0,2),(n,1)]$ at the point $x_{n+k+1}$, for $1\leq k\leq n-1$.
\end{enumerate}
\begin{figure}
 \begin{minipage}[l]{.46\linewidth}
  \centering\epsfig{figure=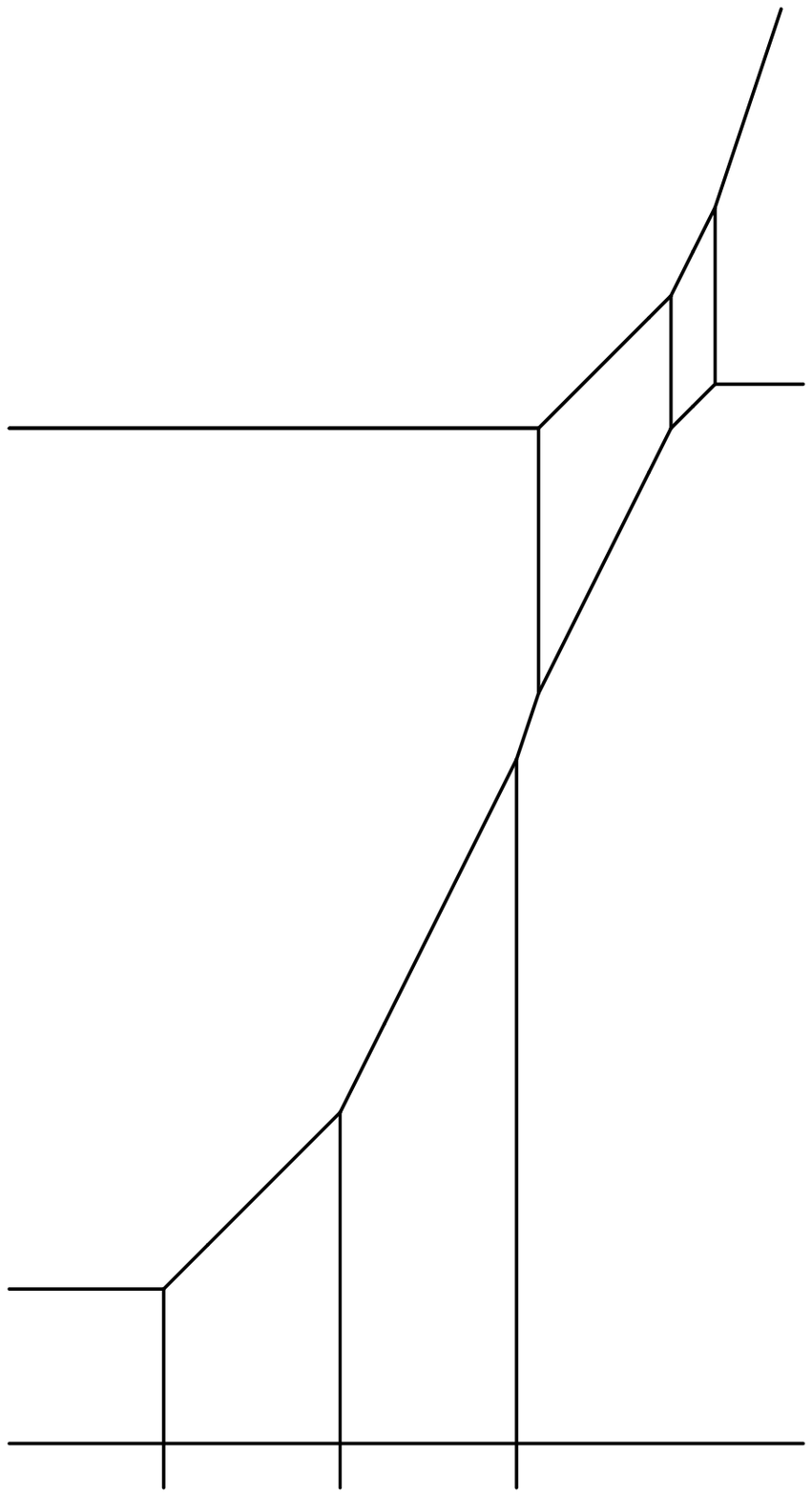,width=17cm}
  \caption{The curve $F_3$ and the curve $E$. \label{courbe4}}
 \end{minipage} \hfill
 \begin{minipage}[l]{.46\linewidth}
  \centering\epsfig{figure=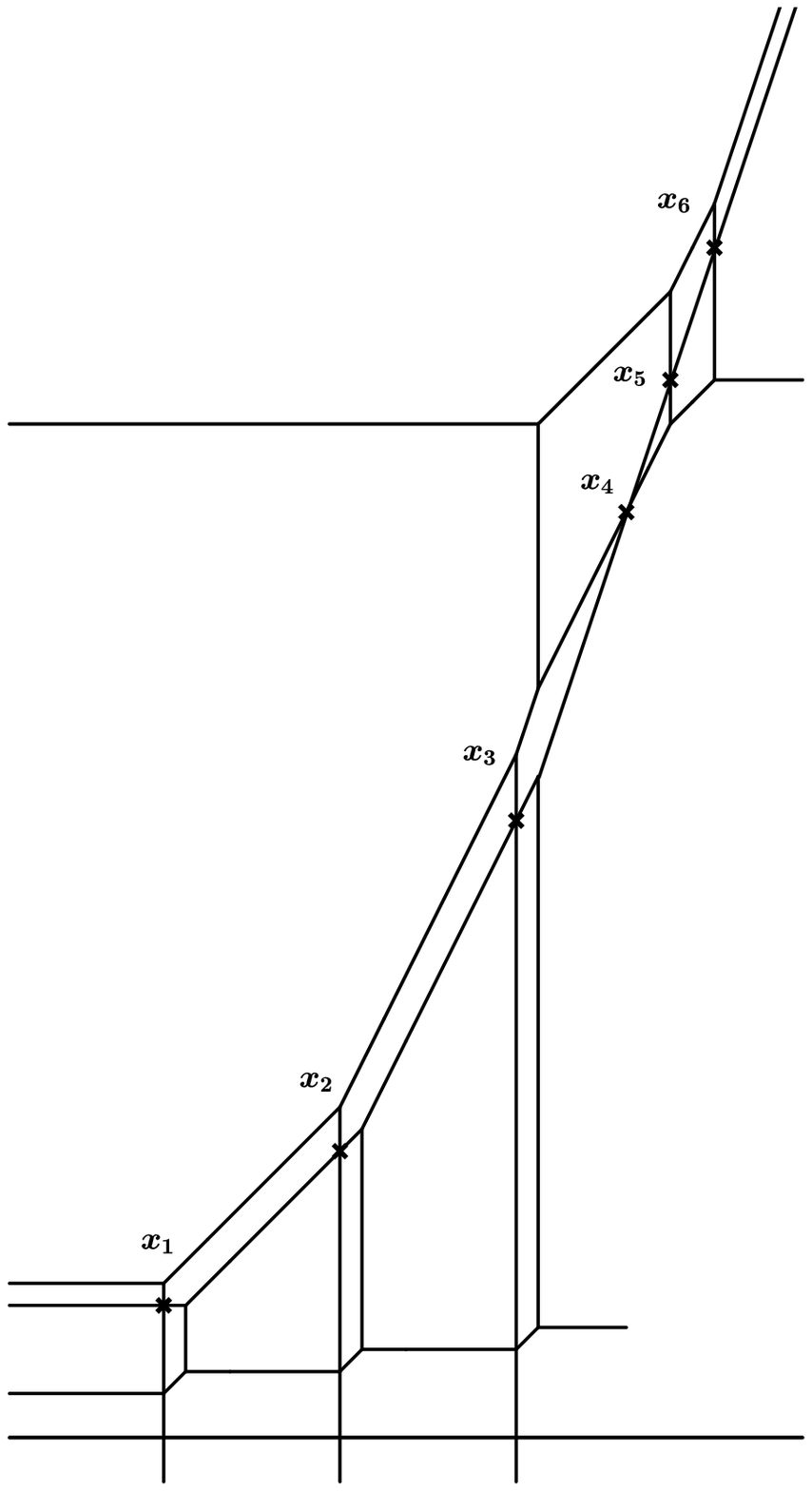,width=17cm}
  \caption{The relative position of the curves $F_3$, $C_3$ and $E$. \label{courbe6}}
 \end{minipage}
\end{figure}
%Up to now, the marked points $x_{n+1},\cdots x_{2n}$ could be anywhere on the edge of $C_n$ dual to the edge $\left[(0,2),(n,1)\right]$ of $\Theta_n$. Fix the marked point $x_{n+1}$ on $C_n$ as the intersection point of the edge of $F_n$ dual to the edge $[(1,1),(n,0)]$ with the  edge of $C_n$ dual to the edge $[(0,2),(n,1)]$. For $1\leq k\leq n-1$, fix the marked point $x_{n+k+1}$ on $C_n$ as the intersection point of the edge of $F_n$ dual to the edge $[(k,1),(k+1,1)]$ with the edge of $C_n$ dual to the edge $[(0,2),(n,1)]$. Denote by $\tilde{C}_n$ the lifting of $\left(C_n,x_1,\cdots ,x_{2n}\right)$ to $X_n$. 
Denote by $f_n$ a tropical polynomial satisfying $V(f_n)=F_n$ and by $g$ a tropical polynomial satisfying $V(g)=E$. Put $h_n^0=``\frac{f_n}{g}"$.
\begin{lemma}
\label{lemma}
\label{uniquelift}
There exists a unique lifting of the curve $\left(C_n,x_1,\cdots ,x_{2n}\right)$ to $\R^2_{h_n^0}$. Moreover, there exists $\lambda_0\in\R$ such that $\tilde{C_n}$ is the lifting of $\left(C_n,x_1,\cdots ,x_{2n}\right)$ to $\R^2_{``\lambda_0 h_n^0"}$.
\end{lemma}
\begin{proof}
Let $e$ be an edge of $C_n$ not contained in $F_n$. Then the lift of $e$ to $\R^2_{h_n^0}$ is uniquely determined.
Consider the subdivision of $\R^2$ induced by the tropical curve $F_n\cup E$. Denote by $F^k$ the face of this subdivision dual to the point $(k,1)$, and by $G^{k}$ the face dual to the point $(k,2)$, for $0\leq k\leq n$. The direction of the face of $\R^2_{h_n^0}$ projecting to $F^{k}$ is generated by the vectors $(0,1,-1)$ and $(1,k,0)$, and the direction of the face of $\R^2_{h_n^0}$ projecting to $G^{k}$ is generated by the vectors $(0,1,0)$ and $(1,k,k)$. It follows from these computations that the lift to $\R^2_{h_n^0}$ of an edge $e$ of $C_n$ not contained in $F_n$ has same direction as the edge of $\tilde{C}_n$ projecting to $e$.  From Remark \ref{remarklift}, we deduce in this case that $\left(C_n,x_1,\cdots ,x_{2n}\right)$ has a unique lifting to $\R^2_{h_n^0}$ and that the result is a vertical translation of $\tilde{C}_n$. Then there exists $\lambda_0\in\R$ such that the lifting of $\left(C_n,x_1,\cdots ,x_{2n}\right)$ to $\R^2_{``\lambda_0 h"}$ is $\tilde{C}_n$.
\end{proof}
Put $h_n=``\lambda_0 h_n^0"$ and $Y_n=\R^2_{``\lambda_0 h"}$.
\begin{lemma}
The surfaces $X_n$ and $Y_n$ intersect transversely and $\tilde{C}_n=X_n\cap Y_n$.
\end{lemma}
\begin{proof}
It follows from Lemma \ref{lemma} that $\tilde{C}_n\subset X_n\cap Y_n$. By construction, any edge of $\tilde{C}_n$ is the transverse intersection of a face of $X_n$ with a face of $Y_n$. So $\tilde{C}_n\subset X_n\cdot Y_n$, where $X_n\cdot Y_n$ denotes the stable intersection of $X_n$ and $Y_n$. Let us compute the number (counted with multiplicity) of infinite edges of $X_n\cdot Y_n$ in a given direction. Denote by $\Delta(X_n)$ the Newton polytope of $X_n$, and denote by $\Delta(Y_n)$ the Newton polytope of $Y_n$. One has
$$
\Delta(X_n)=Conv\left((0,0,0),(n,0,0),(0,1,0),(0,0,1)\right),
$$
and
$$
\Delta(Y_n)=Conv\left((0,0,0),(n,0,0),(n,1,0),(0,2,0),(0,0,1),(1,0,1)\right),
$$
see Figure \ref{NP1} and Figure \ref{NP2}.
\begin{figure}
 \begin{minipage}[l]{.46\linewidth}
  \centering\epsfig{figure=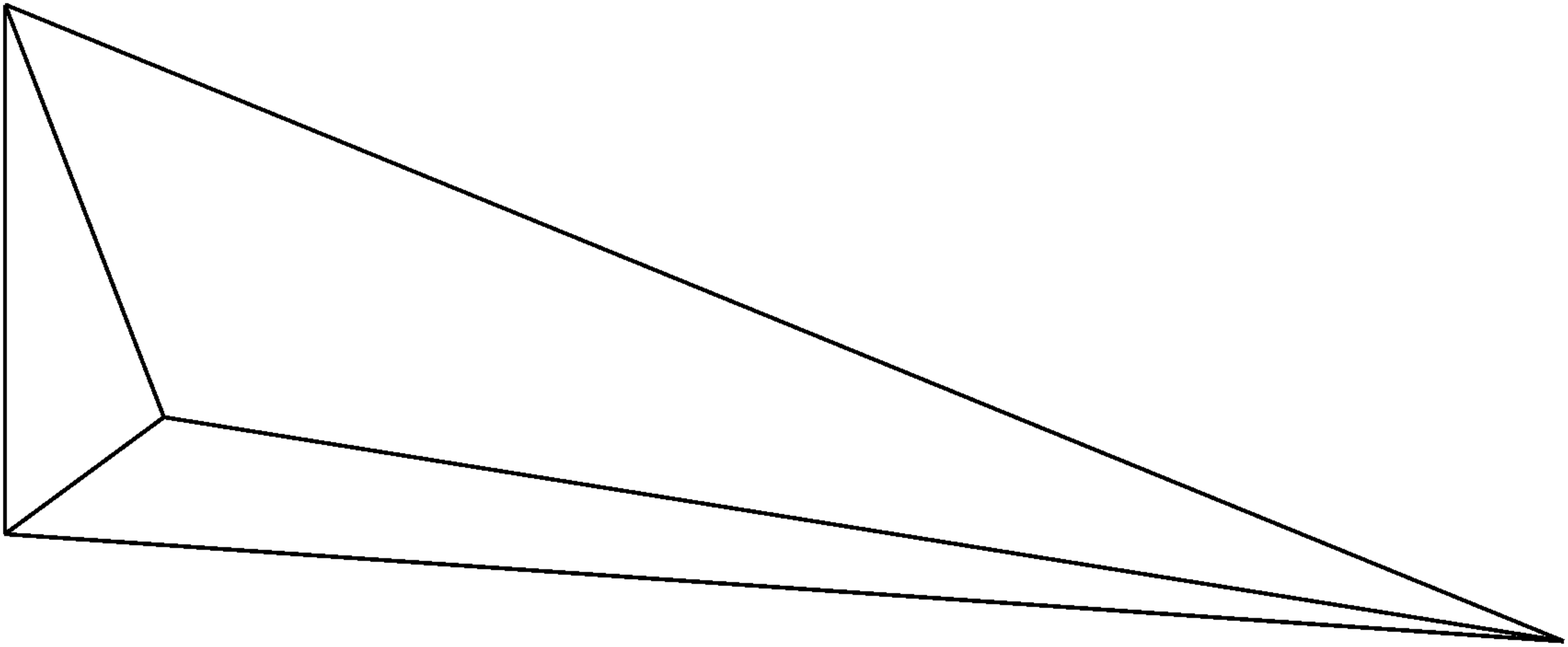,width=8cm}
  \caption{The Newton polytope of $X_n$. \label{NP1}}
 \end{minipage} \hfill
 \begin{minipage}[l]{.46\linewidth}
  \centering\epsfig{figure=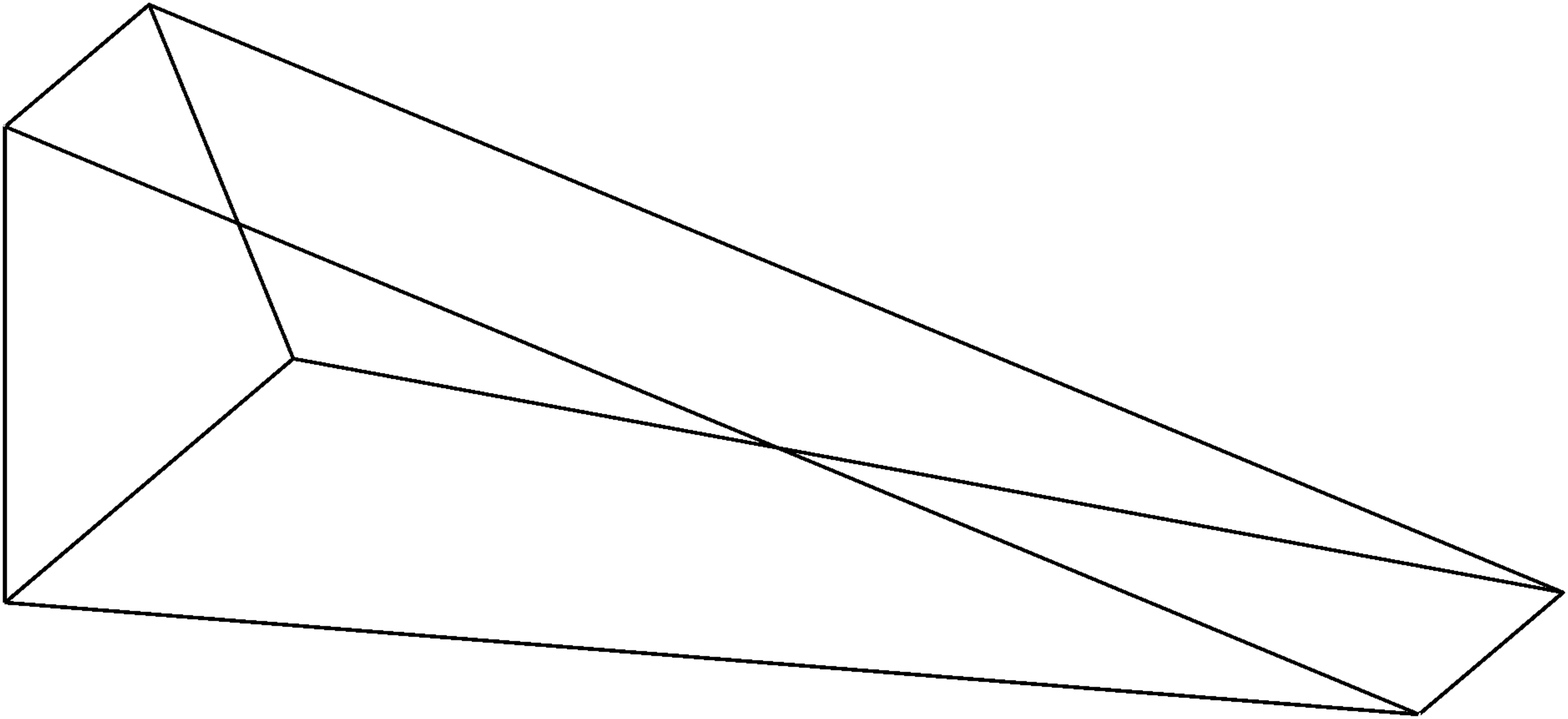,width=8cm}
  \caption{The Newton polytope of $Y_n$. \label{NP2}}
 \end{minipage}
\end{figure}
By considering faces of the Minkowsky sum 
$$
\Delta(X_n)+\Delta(Y_n)=\left\lbrace a+b\mid a\in \Delta(X_n) \mbox{ and } b\in\Delta(Y_n) \right\rbrace,
$$ one can see that $X_n\cdot Y_n$ has two edges (counted with multiplicity) of direction $(-1,0,0)$, $n$ edges (counted with multiplicity) of direction $(0,-1,0)$, $2n$ edges (counted with multiplicity) of direction $(0,0,-1)$, one edge of direction $(1,0,n)$ and one edge of direction $(1,n,n)$. Since the curve $\tilde{C}_n$ has also two edges of direction $(-1,0,0)$, $n$ edges of direction $(0,-1,0)$, $2n$ edges of direction $(0,0,-1)$, one edge of direction $(1,0,n)$ and one edge of direction $(1,n,n)$, we conclude that  $\tilde{C}_n = X_n\cdot Y_n=X_n\cap Y_n$. 
\end{proof}
\section{Real phases on $X_{n}$ and $Y_{n}$}
\label{construction2}
%In this section, we define real phases $\varphi_{\tilde{C}_n}$ on $\tilde{C}_n$ and $\varphi_{P_n}$ on $X_n$. Then we prove that one can define also real phase $\varphi_{h_n}$ on $Y_n$ such that the following compatibily condition is satisfied.
%$$
%\left(\R X_n\right)_{\varphi_{P_n}}\cap
%\left(\R Y_n\right)_{\varphi_{h_n}}=
%\left(\R \tilde{C}_n\right)_{\varphi_{\tilde{C}_n}}.
%$$
\label{sectionrealphases}
Consider the Harnack phase $\varphi_{D_n}$ on $D_n$ and the Harnack phase $\varphi_{C_n}$ on $C_n$ (see Definition \ref{Harnackphase}). We depicted the real part of $D_3$ on Figure \ref{courbe7} and the real part of $C_3$ on Figure \ref{courbe8}. Notice that for the real phase $\varphi_{C_n}$ on $C_n$, every edge of $C_n$ containing a marked point is equipped with the sign $(-,+)$. Consider the $2n$ marked points $r_1,\cdots r_{2n}$ on $\R C_n$, where $r_i$ is the symmetric copy of $x_i$ in the quadrant $\R_{-}^*\times\R_{+}^*$ (see Figure \ref{courbe8} for the case $n=3$).

\begin{figure}
 \begin{minipage}[l]{.46\linewidth}
  \centering\epsfig{figure=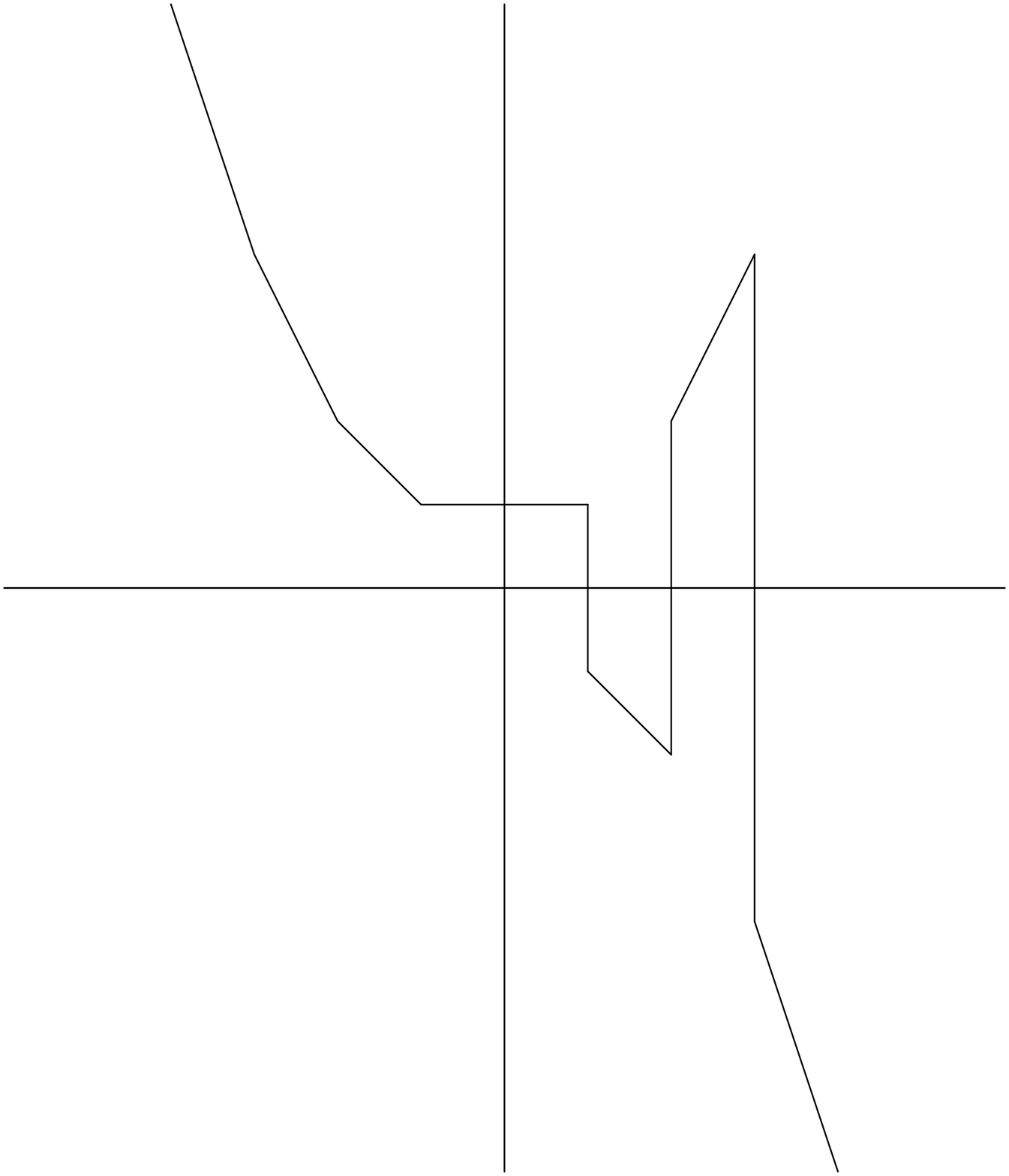,width=11cm}
  \caption{The real part $\R D_3$. \label{courbe7}}
 \end{minipage} \hfill
 \begin{minipage}[l]{.46\linewidth}
  \centering\epsfig{figure=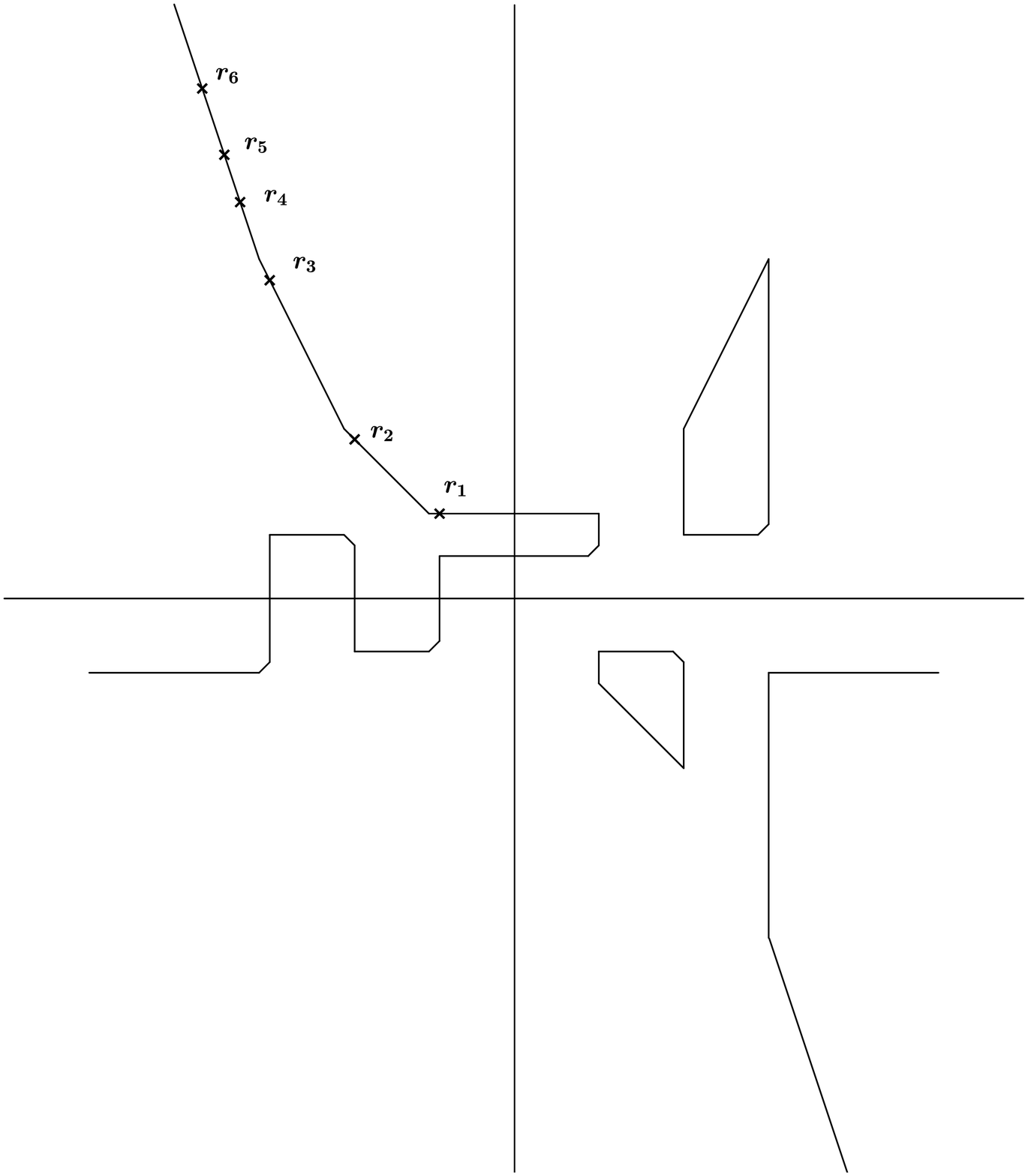,width=11cm}
  \caption{The real part $\R C_3$. \label{courbe8}}
 \end{minipage}
\end{figure}
%A real phase $\varphi_{\tilde{C}_n}$ on $\tilde{C}_n$ is called compatible with $\varphi_{C_n}$ if 
%$$
%\pi\left(\R\tilde{C}_n\right)_{\varphi_{\tilde{C}_n}}=
%\left(\R C_n\right)_{\varphi_{C_n}},
%$$
%where $\pi$ denotes the projection forgetting the last factor. A real phase $\varphi_{\tilde{C}_n}$ on $\tilde{C}_n$ compatible with $\varphi_{C_n}$ equip any edge of $\left(\R C_n\right)_{\varphi_{C_n}}$ with a triple of sign. A marked point $v$ on $\left(\R C_n\right)_{\varphi_{C_n}}$ is called essential if the last sign of the adjacent edges to $v$ are different.
\begin{lemma}
\label{lemmarealphases}
There exist a real phase $\varphi_{X_n}$ on $X_n$ and a real phase $\varphi_{Y_n}$ on $Y_n$ such that the following conditions are satisfied:
\begin{enumerate}
\item $\R D_n$ is the projection of the union of all vertical faces of $\R X_n$.\label{cond1}
\item $\pi^\R \left(\R X_n\cap \R Y_n\right)=\R C_n$.\label{cond2}
\item An edge $e$ of $\R X_n\cap \R Y_n$ is unbounded in direction $(0,0,-1)$ if and only if \linebreak $\pi^\R (e)\subset\left \lbrace r_1,\cdots ,r_{2n} \right\rbrace$.\label{cond3}
%\item $\R \tilde{C}_n\subset\R X_n$.\label{cond4}
\end{enumerate}  
\end{lemma}
\begin{proof}
Denote by $\delta_{D_n}$ the Harnack distribution of signs at the vertices of the subdivision of $\Delta_n$ dual to $D_n$. By definition, $\varphi_{D_n}$ is compatible with $\delta_{D_n}$. Complete the distribution of signs $\delta_{D_n}$ to a distribution of signs $\delta_{X_n}$ at the vertices of the dual subdivision of $X_n$ by choosing an arbitrary sign for the vertex $(0,0,1)$. Define $\varphi_{X_n}$ to be the real phase on $X_n$ compatible with $\delta_{X_n}$. By construction, $\varphi_{X_n}$ satisfies Condition \ref{cond1} of Lemma \ref{lemmarealphases}. Define a real phase on $Y_n$ as follows. Denote by $\mathcal{G}$ the set of all faces of $Y_n$ containing an edge of $\tilde{C}_n$.
\\We first define a real phase $\varphi_{Y_n,F}$ on any face $F\in\mathcal{G}$.
%Let us prove that Conditions \ref{cond2}, \ref{cond3} and \ref{cond4} determine uniquely the real phase $\varphi_{\tilde{C}_n}$ on $\tilde{C}_n$. 
Consider, as in the proof of Lemma \ref{uniquelift}, the subdivision of $\R^2$ induced by the tropical curve $F_n\cup E$.  Denote by $F^k$ the face of this subdivision dual to the point $(k,1)$, and by $G^{k}$ the face dual to the point $(k,2)$, for $0\leq k\leq n$. Denote by $\tilde{F}^k$ the face of $Y_n$ such that $\pi^\R(\tilde{F}^k)=F^k$ and by $\tilde{G}^k$ the face of $Y_n$ such that $\pi^\R(\tilde{G}^k)=G^k$, for $0\leq k\leq n$. Denote by $f_k$ the edge of $F_n$ containing $x_k$ and by $\tilde{f}_k$ the vertical face of $Y_n$ projecting to $f_k$, for $1\leq k\leq 2n$. One can see that all the edges of $\tilde{C}_n$ are contained in the union of all faces $\tilde{F}^k$, $\tilde{G}^k$ and $\tilde{f}_k$. For any $0\leq k\leq n$, one can see that the face $\tilde{F}^k$ contains an edge $e$ of $\tilde{C}_n$ such that $e$ is contained in a non-vertical face $F_e$ of $X_n$. Denote by $(\varepsilon_1,\varepsilon_2)$ a component of the real phase $\varphi_{C_n}$ on $\pi^\R(e)$. Since $F_e$ is non-vertical, there exists a unique sign $\varepsilon_3$ such that $(\varepsilon_1,\varepsilon_2,\varepsilon_3)$ is a component of the real phase $\varphi_{X_n}$ on $F_e$. Define the real phase $\varphi_{Y_n}$ on $\tilde{F}^k$ to contain $(\varepsilon_1,\varepsilon_2,\varepsilon_3)$ as a component. Condition \ref{cond3} of Lemma \ref{lemmarealphases} determines the real phase $\varphi_{Y_n}$ on $\tilde{f_k}$, for any $1\leq k\leq 2n$. Since the three faces $\tilde{F}^n$, $\tilde{G}^1$ and $\tilde{f}_{n+1}$ are adjacent, it follows from the definition of a real phase that the phase on $\tilde{G}^1$ is determined from the phase on $\tilde{F}^n$ and the phase on $\tilde{f}_{n+1}$. Since the three faces $\tilde{G}^{k+1}$, $\tilde{G}^k$ and $\tilde{f}_{n+k+1}$ are adjacent, for any $1\leq k\leq n-1$, it follows from the definition of a real phase that the phase on $\tilde{G}^{k+1}$ is determined from the phase on $\tilde{G}^k$ and the phase on $\tilde{f}_{n+k+1}$. By induction, it determines the real phase $\varphi_{Y_n}$ on $\tilde{G}^k$, for $1\leq k\leq n$. 
\\ We now extend the definition of $\varphi_{Y_n}$ to all faces of $Y_n$. Consider the set of edges $\mathcal{E}$ of the dual subdivision of $Y_n$ such that $e\in\mathcal{E}$ if and only if $e$ is dual to an element in $\mathcal{G}$. Consider $\mathcal{V}$ the set of vertices of edges in $\mathcal{E}$. As in Lemma \ref{lemmarealphase}, one can consider a distribution of signs on $\mathcal{V}$ compatible with the real phase on $\mathcal{G}$. Extend arbitrarily this distribution to all vertices of the dual subdivision of $Y_n$ and consider the real phase $\varphi_{Y_n}$ on $Y_n$ compatible with the extended distribution of signs. By construction, the real phases $\varphi_{X_n}$ and $\varphi_{Y_n}$ satisfy the Conditions \ref{cond1}, \ref{cond2} and \ref{cond3} of Lemma \ref{lemmarealphases}.  
\end{proof}
Put
$$
\R \tilde{C}_n=\R X_n\cap \R Y_n.
$$

Consider, as explained in Section \ref{strategy}, a family of real algebraic curves $\mathcal{D}_{n,t}$ approximating 
$(D_n,\varphi_{D_n})$, a family of real algebraic surfaces $\mathcal{X}_{n,t}$ approximating 
$(X_n,\varphi_{X_n})$ and a family of real algebraic surfaces $\mathcal{Y}_{n,t}$ approximating 
$(Y_n,\varphi_{Y_n})$. Put $\tilde{\mathcal{C}}_{n,t}=\mathcal{X}_{n,t}\cap \mathcal{Y}_{n,t}$. 

For every $t$, put $\mathcal{C}_{n,t}=\pi^\C\left(\tilde{\mathcal{C}}_{n,t}\right)$.
Consider $\overline{\R X_n}$ the partial compactification of $\R X_n$ in $(\R^*)^2\times \R$ and $\overline{\R \mathcal{X}_{n,t}}$ the partial compactification of $\R \mathcal{X}_{n,t}$ in $(\R^*)^2\times \R$, for any $t$.
% and $\overline{\R \tilde{\mathcal{C}}_{n,t}}$ the partial compactification of $\R \tilde{\mathcal{C}}_{n,t}$ in $(\R^*)^2\times \R$. 
 One has $\overline{\R X_n}\cap\left((\R^*)^2\times \lbrace 0\rbrace\right)=\R D_n$ and $\overline{\R  \mathcal{X}_{n,t}}\cap\left((\R^*)^2\times \lbrace 0\rbrace\right)=\R \mathcal{D}_{n,t}$.
Then, it follows from Theorem \ref{patchtropcomp} that for $t$ large enough, one has the following homeomorphism of pairs:
\begin{equation}
\left( \overline{\R\mathcal{X}_{n,t}},\R \mathcal{D}_{n,t}\cup\R\tilde{\mathcal{C}}_{n,t}\right)\simeq \left(\overline{\R X_n},\R D_n\cup\R\tilde{C}_{n}\right).
\label{homeo1}
\end{equation}

Moreover, the map $\pi^\R$ gives by restriction a bijection from $\overline{\R\mathcal{X}_{n,t}}$ to $(\R^*)^2$ fixing $\R \mathcal{D}_{n,t}$ and sending $\R\tilde{\mathcal{C}}_{n,t}$ to $\R\mathcal{C}_{n,t}$. The map $\pi^\R$ and the homeomorphism (\ref{homeo1}) give rise to the following homeomorphism of pairs, for $t$ large enough:
\begin{equation}
\left((\R^*)^2,\R \mathcal{D}_{n,t}\cup \R\mathcal{C}_{n,t}\right)\simeq
\left(\overline{\R X_n},\R D_n\cup\R\tilde{C}_{n}\right).
\end{equation}

It remains to describe the pair $\left(\overline{\R X_n},\R D_n\cup\R\tilde{C}_{n}\right)$. One has $\pi^\R(\overline{\R X_n})=(\R^*)^2$ and for any $x\in \R D_n$, $(\pi^\R)^{-1}(x)$ is an interval of the form $\lbrace x\rbrace\times \left[-h,h\right]$, see Figure \ref{surface1} for a local picture.
\begin{remark}
\begin{figure}
 \begin{minipage}[l]{.46\linewidth}
  \centering\epsfig{figure=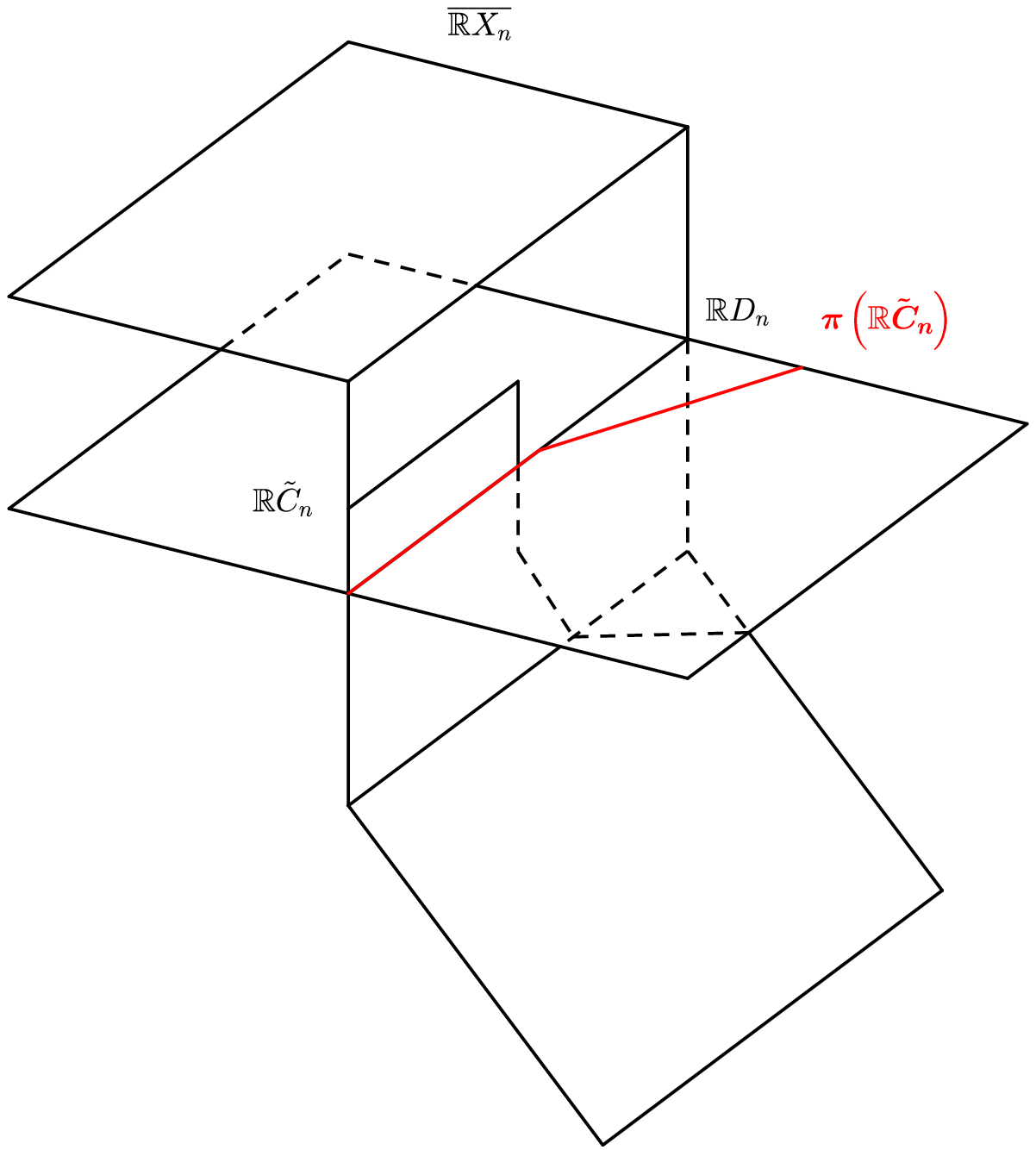,width=12cm}
  \caption{The curve $\R D_n\cup\R \tilde{C}_n$ in the surface $\overline{\R X_n}$ and the projection $\pi^\R(\R \tilde{C}_n)$. \label{surface1}}
 \end{minipage} \hfill
 \begin{minipage}[l]{.46\linewidth}
  \centering\epsfig{figure=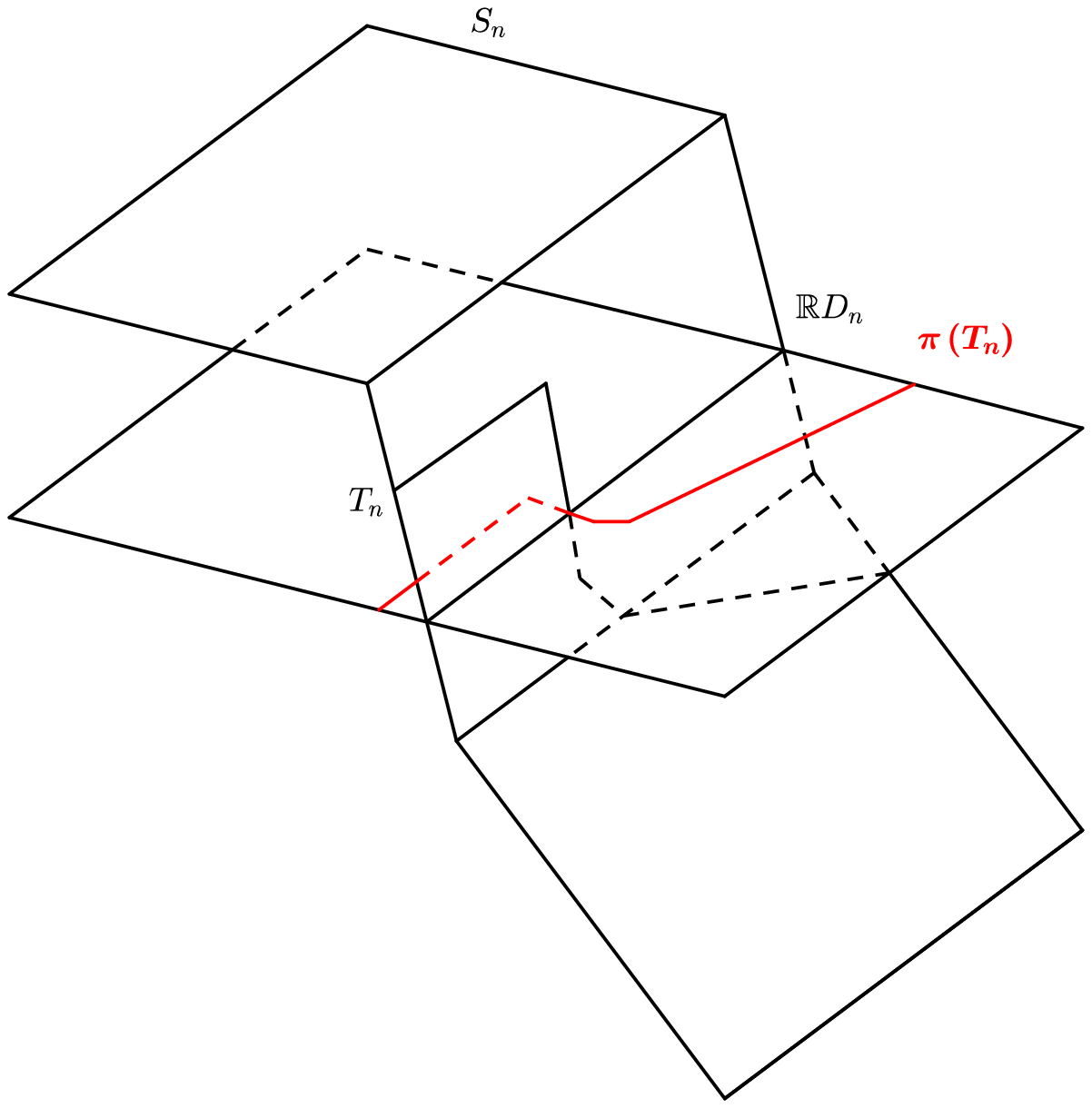,width=12cm}
  \caption{The curve $\R D_n\cup T_n$ in the surface $S_n$ and the projection $\pi^\R(T_n)$\label{surface2}}
 \end{minipage}
\end{figure}

The set $(\pi^\R)^{-1}(\R D_n)$ can be seen as a tubular neighborhood of $\R D_n$ in $\R X_n$. 
\end{remark}
Outside of $(\pi^\R)^{-1}(\R D_n)$, the map $\pi\vert_{\overline{\R X_n}}$ is bijective. Let $V$ be a small tubular neighborhood of $\R D_n$ in $(\R^*)^2$. Perturb slightly the pair $(\overline{\R X_n},\R\tilde{C}_n)$ inside $(\pi^\R)^{-1}(V)$ to produce a pair $(S_n,T_n)$, such that 
\begin{itemize}
\item $(\overline{\R X_n},\R\tilde{C}_n)\simeq (S_n,T_n)$,
\item $\R D_n\subset S_n$,
\item $\pi^\R$ defines an homeomorphism from $S_n$ to $(\R^*)^2$,
\end{itemize}
see Figure \ref{surface1} and Figure \ref{surface2} for a local picture. One obtains the following homeomorphism of pairs, for $t$ large enough:
\begin{equation}
\left((\R^*)^2,\R \mathcal{D}_{n,t}\cup \R\mathcal{C}_{n,t}\right)\simeq
\left((\R^*)^2,\R D_n\cup\pi^\R(T_n)\right).
\end{equation}
By construction, the curve $\pi^\R(T_n)$ is a small perturbation of $\R C_n$, intersecting the curve $\R D_n$ transversely and only at the marked points $r_1,\cdots, r_{2n}$, see Figure \ref{courbe10} and Figure \ref{courbe9} for the case $n=3$. Then, for $t$ large enough, the chart of the reducible curve $\mathcal{D}_{n,t}\cup\mathcal{C}_{n,t}$ is homeomorphic to the chart depicted in Figure \ref{courbebrugallered}.

\begin{figure}
 \begin{minipage}[l]{.46\linewidth}
  \centering\epsfig{figure=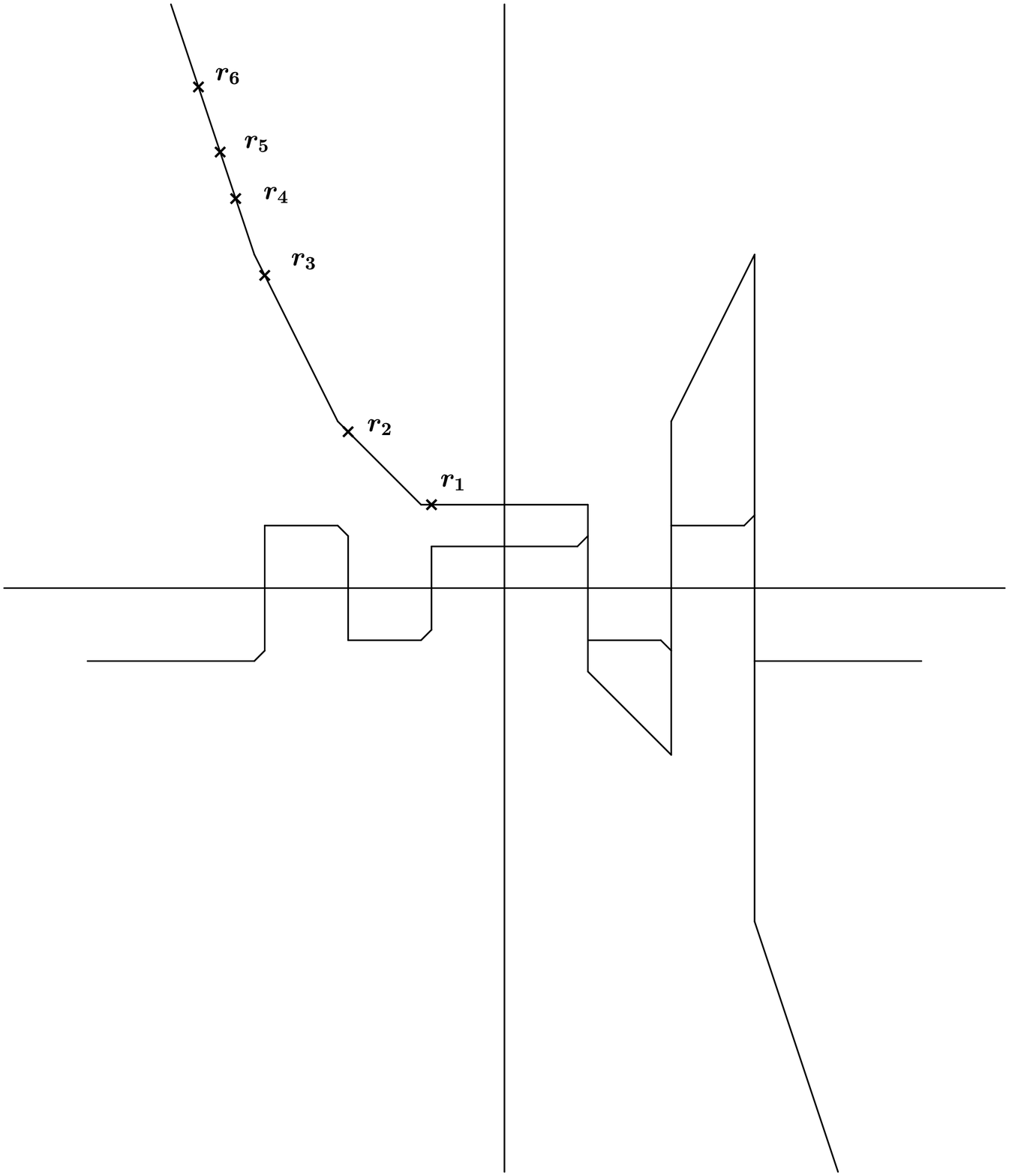,width=12cm}
  \caption{$\R C_3\cup \R D_3$. \label{courbe10}}
 \end{minipage} \hfill
 \begin{minipage}[l]{.46\linewidth}
  \centering\epsfig{figure=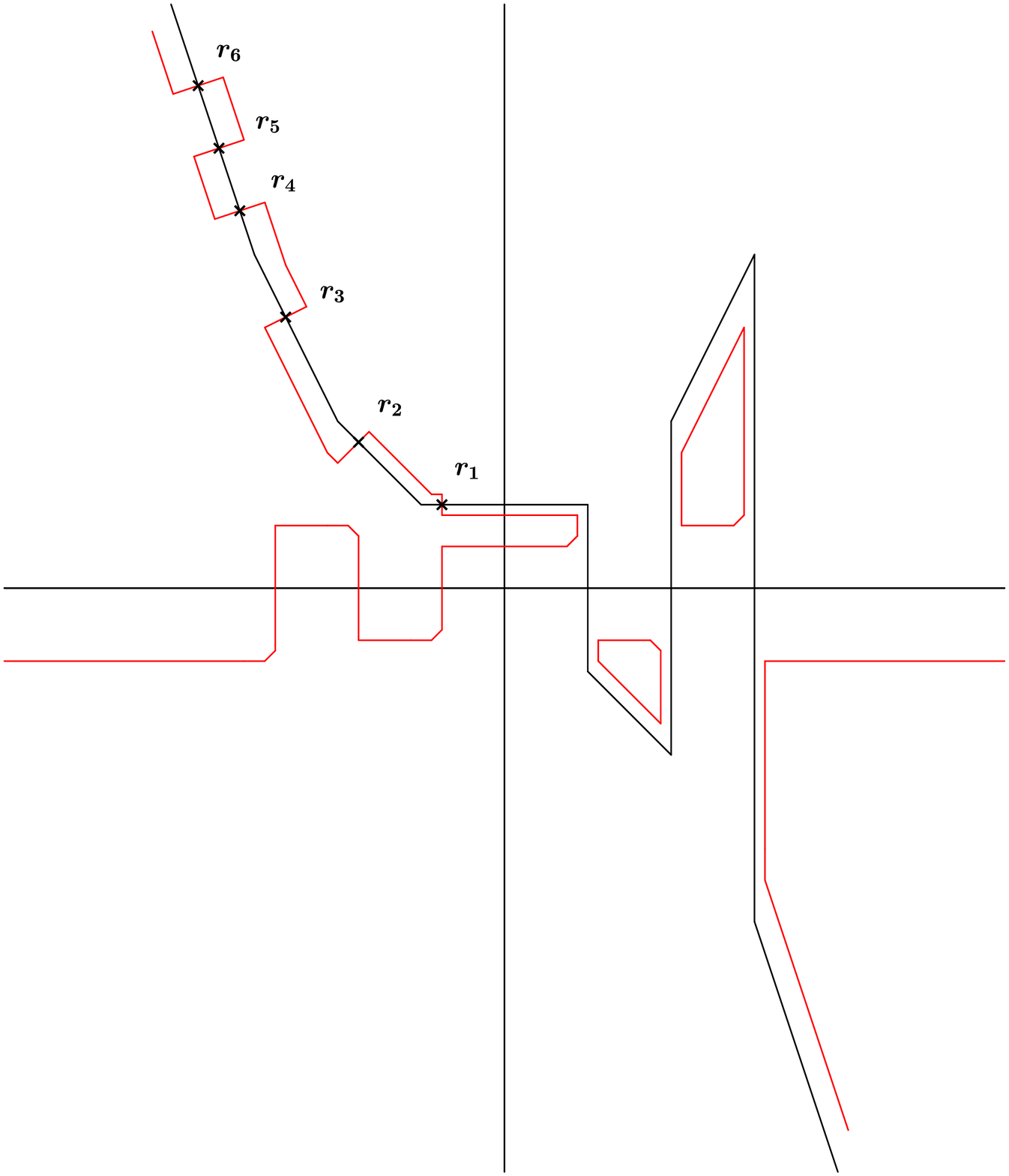,width=12cm}
  \caption{$\pi(T_3)\cup\R D_3$ \label{courbe9}}
 \end{minipage}
\end{figure}

\bibliographystyle{alpha}
\bibliography{biblio}
\end{document}